\documentclass[12pt,a4paper,reqno]{amsart}

\usepackage{a4wide}

\usepackage[T1]{fontenc}
\usepackage{lmodern}
\usepackage[utf8]{inputenc}

\usepackage{amsmath, amsthm, amsfonts, amssymb, mathptmx}

\usepackage{mathrsfs}

\usepackage{microtype}

\usepackage{graphicx}
\usepackage{subcaption}
\usepackage{epstopdf}

\usepackage{hyperref}
\usepackage{cleveref}
\usepackage{cite}
\usepackage{enumerate}
\usepackage{fancyhdr}

\crefname{equation}{}{}
\Crefname{equation}{Equation}{Equations}

\theoremstyle{plain}
\newtheorem{thm}{Theorem}[section]
\crefname{thm}{theorem}{theorems}
\Crefname{thm}{Theorem}{Theorems}

\newtheorem*{mainthm}{Main Theorem}

\newtheorem{lemma}[thm]{Lemma}
\crefname{lemma}{lemma}{lemmas}
\Crefname{lemma}{Lemma}{Lemmas}

\newtheorem{prop}[thm]{Proposition}
\crefname{prop}{proposition}{propositions}
\Crefname{prop}{Proposition}{Propositions}

\newtheorem{corr}[thm]{Corollary}

\theoremstyle{definition}
\newtheorem{dfn}[thm]{Definition}
\theoremstyle{remark}
\newtheorem*{remark}{Remark}

\numberwithin{equation}{section}

\DeclareMathAlphabet{\mathcal}{OMS}{cmsy}{m}{n}

\newcommand{\Torus}{\mathbb{T}}

\newcommand{\K}{\beta}

\title[Power law asymptotics in the quasi-periodically forced quadratic family]{Power law asymptotics in the creation of strange attractors in the quasi-periodically forced quadratic family}
\author{Thomas Ohlson Timoudas}

\begin{document}

\begin{abstract}
Let $\Phi$ be a quasi-periodically forced quadratic map, where the rotation constant $\omega$ is a Diophantine irrational.
A strange non-chaotic attractor (SNA) is an invariant (under $\Phi$) attracting graph of a nowhere continuous measurable function $\psi$ from the circle $\mathbb{T}$ to $[0,1]$.

This paper investigates how a smooth attractor degenerates into a strange one, as a parameter $\K$ approaches a critical value $\K_0$,
and the asymptotics behind the bifurcation of the attractor from smooth to strange.
In our model, the cause of the strange attractor is a so-called torus collision, whereby an attractor collides with a repeller.

Our results show that the asymptotic minimum distance between the two colliding invariant curves decreases linearly in the parameter $\K$,
as $\K$ approaches the critical parameter value $\K_0$ from below.

Furthermore, we have been able to show that the asymptotic growth of the supremum of the derivative of the attracting graph is asymptotically bounded from both sides
by a constant times the reciprocal of the square root of the minimum distance above.
\end{abstract}

\maketitle

\section{Introduction}

In recent decades much attention has been directed towards the investigation of \textit{strange attractors}, attractors with a fractal or highly discontinuous structure, and how they appear.
Even to this date, most of the work is of a numerical nature, and there are only few rigorous results about them.
Here, we will present some rigorous results concerning certain asymptotics in the bifurcations of a smooth attractor into a strange one.

The term \textit{strange attractor} was coined in the early 70's in \cite{RuelleTurbulence},
where the authors made a connection between turbulence and strange attractors.
More than a decade later, \cite{GOPY} introduced the concept of a \textit{strange non-chaotic attractor} (SNA for short), strange attractors with non-positive Lyapunov exponents.

Some of the earliest constructions of SNA's can be found in \cite{MillionSNA1, MillionSNA2, HermanMinorer, RusselJohnsonSNA},
though they pre-dated the actual term (and seemed largely unknown to the early researchers on SNA's).
In the beginning, the advances were mainly numerically supported, and the standing question was whether they actually exist at all (and what they actually are).

The next question, if they should indeed exist, presented itself: could they appear outside of abstract models,
concocted in the minds of mathematicians? That is, are they of any physical relevance - can they be observed in nature?
In fact, there has been experimental evidence of SNA's in certain physical systems (see for instance \cite{ExperimentalSNA}).

In physics, it is common to have one system driven by another one. This is called forcing. The most well-known type is periodic forcing.
There is however another important, but much less understood, mode of forcing called \textit{quasi-periodic}:
\begin{align}\label{QPFSystem}
\left\{
\begin{array}{l}
\theta_{n+1} = \theta_n + \omega \\
x_{n+1} = f(\theta_n, x_n),
\end{array}
\right.
\end{align}
where $x \in \mathbb{R}$, $\theta$ lies in the circle $\mathbb{T} = [0,1]$, where 0 and 1 are identified, $\omega$ is irrational, and $f$ is smooth.
If the Lyapunov exponent of this system in the $x$-direction is negative for every $(\theta,x) \in \mathbb{T} \times (0,1)$,
then it has a continuous attracting invariant curve $\psi: \mathbb{T} \to [0,1]$ which is as smooth as $f$ (see \cite{StarkRegularityQPF}).

Already from the very beginning, the study of SNA's has been intimately linked to the study of quasi-periodically forced (one-dimensional) dynamical systems.
One early paper establishing the existence of SNA's is \cite{BezhOselSNA}.
Another early paper, \cite{KellerSNA}, proves the existence of SNA's in a certain class of pinched\footnote{In a pinched system, one of the fibres in the $x$-direction is identically mapped to $x=0$ (the invariant curve)}
quasi-periodic systems (building on the work in \cite{GOPY}). Pinched systems are also studied in \cite{HaroSNAPinched}.

Following \cite{AlsedaMis, BjerkSNA}, we will adopt the definition of an SNA as being the attracting graph of a measurable curve $\psi: \mathbb{T} \to [0,1]$ which is a.e. discontinuous.
We allow for the possibility of an attractor to attract only a set of points of positive measure, rather than an open neighbourhood of the curve (see \cite{MilnorAttractor}).

Having answered the question of existence in the affirmative, we now wish to understand how SNA's appear; in particular the kinds of bifurcations leading to their formation.
In this paper, we have obtained very precise asymptotics involved in one type of bifurcation for certain quasi-periodically forced logistic maps (an extension of the one considered in \cite{BjerkSNA, BjerkSNA2})

\begin{align}\label{QPFLM}
\left\{
\begin{array}{l}
\theta_{n+1} = \theta_n + \omega \\
x_{n+1} = (\frac32 + \K a(\theta_n)) x_n(1 - x_n),
\end{array}
\right.
\end{align}
modeled on the cylinder $\mathbb{T} \times [0,1]$, for certain $a(\theta)$ (see below), where $0 \leq a(\theta) \leq \frac52$, and $\omega$ is a Diophantine irrational.
For parameter values $0 \leq \K < 1$, we will show that the system has a smooth attracting curve (attracting $\mathbb{T} \times (0,1)$) with negative Lyapunov exponent in the $x$-direction.
However, as proved in \cite{BjerkSNA, BjerkSNA2}, the system has an SNA for $\K = 1$, which is dense in a 2-dimensional surface.
The construction is achieved without pinching (the method used in \cite{KellerSNA}).

The cause for the appearance of the SNA is a collision between the attractor and the invariant (repelling) curve at $x=0$.
In the literature, this is called a torus-collision, a well-known cause of SNA's (see for instance \cite{JorbaOldNewResultsSNA,HaroPuigSNAHarper}).
As the tori approach one another, the attractor starts "wrinkling" (the derivative increases) until it finally "shatters" to form a strange attractor.

The reason we have chosen to study the logistic family is simply because it is one of the most well-studied dynamical systems, and much is known about them (see \cite{BenedicksLogistic, LyubichLogistic, AvilaLogistic}).
The map $a(\theta)$ was chosen to be close to $0$ for most values of $\theta \in \mathbb{T}$,
in order to ensure that orbits stay close to $\frac13$ (the fixed point for $\frac32 x(1-x)$).
However, at two values $\theta = 0$ and $\theta \approx \omega$, $a(\theta)$ suddenly peaks. When $\K = 1$, the peaks reach $4$ (see \cref{CGraph1}),
producing a chain $\frac12 \mapsto 1 \mapsto 0$ (the torus collision) for a certain value of $\theta = \alpha_c$. When $0 \leq \K < 1$, the peaks are linearly scaled by that factor.

\begin{figure}[h]
    \centering
    \begin{subfigure}[b]{0.45\textwidth}
    	\includegraphics[width=\textwidth]{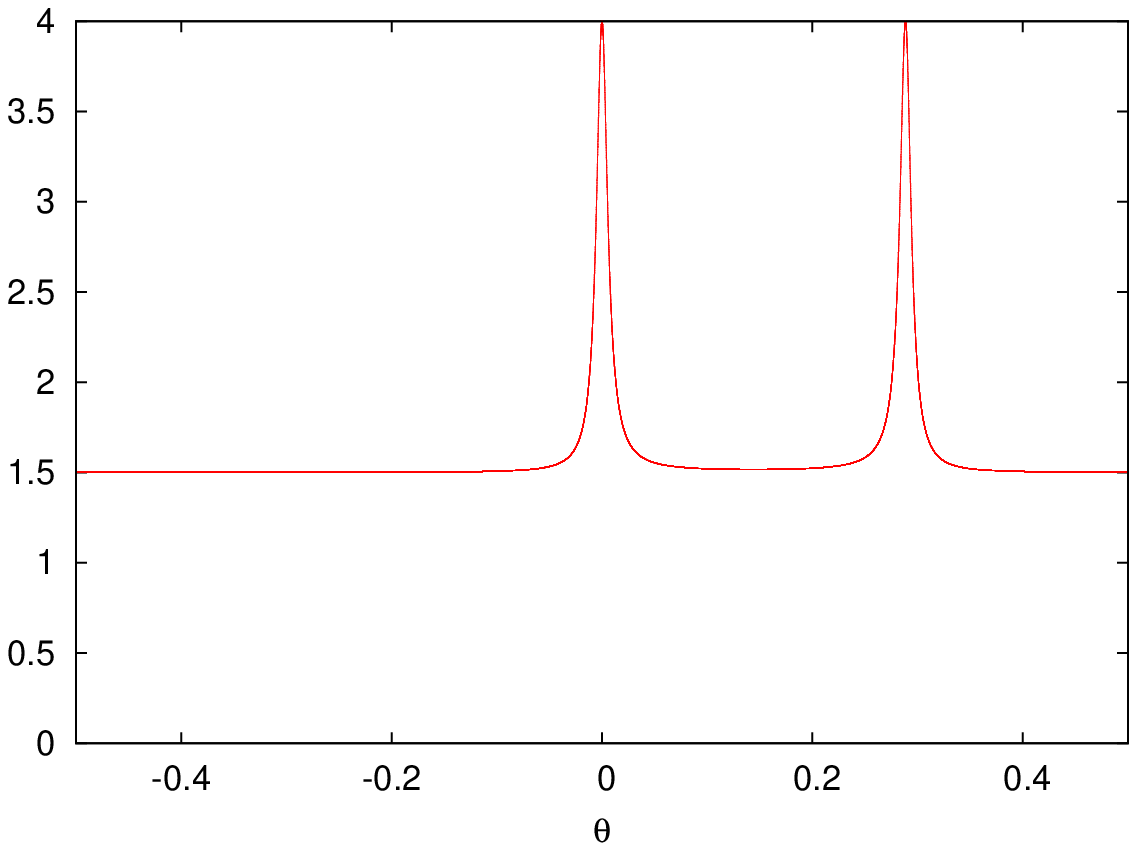}
    	\caption{$\frac32 + \K a(\theta)$ when $\K = 1$.}\label{CGraph1}
    \end{subfigure}
	\begin{subfigure}[b]{0.45\textwidth}
    	\includegraphics[width=\textwidth]{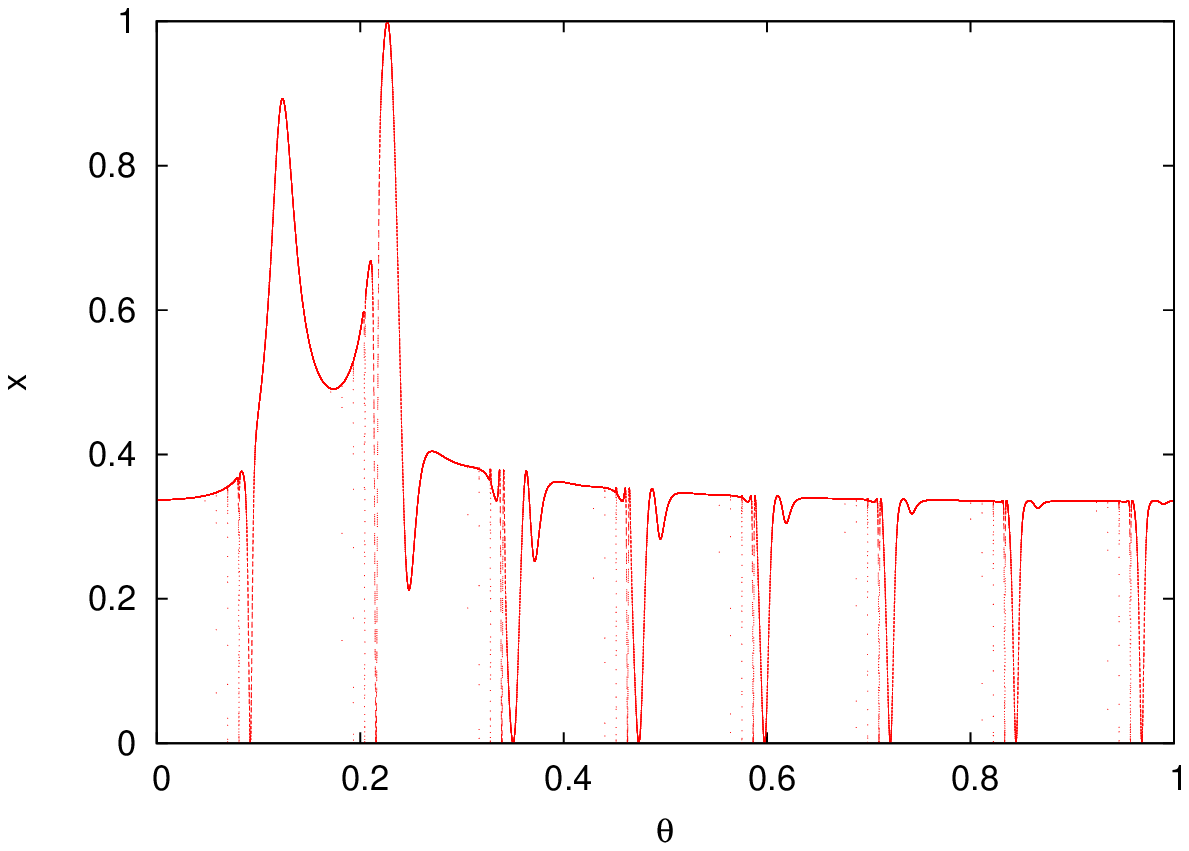}
    	\caption{The attractor when $\K = 1$.}\label{SNAGraph1}
    \end{subfigure}
    \\
    \begin{subfigure}[b]{0.45\textwidth}
    	\includegraphics[width=\textwidth]{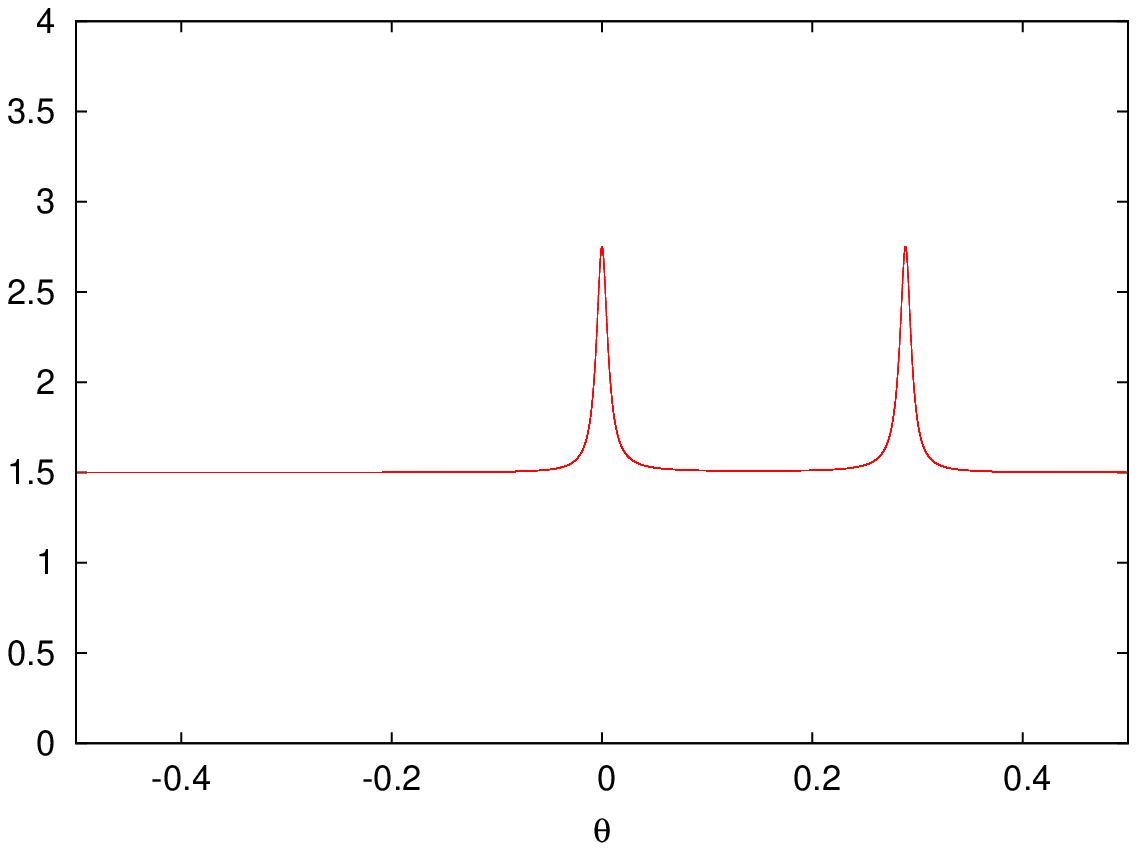}
    	\caption{$\frac32 + \K a(\theta)$ when $\K = 0.5$.}\label{CGraph0_5}
    \end{subfigure}
	\begin{subfigure}[b]{0.45\textwidth}
    	\includegraphics[width=\textwidth]{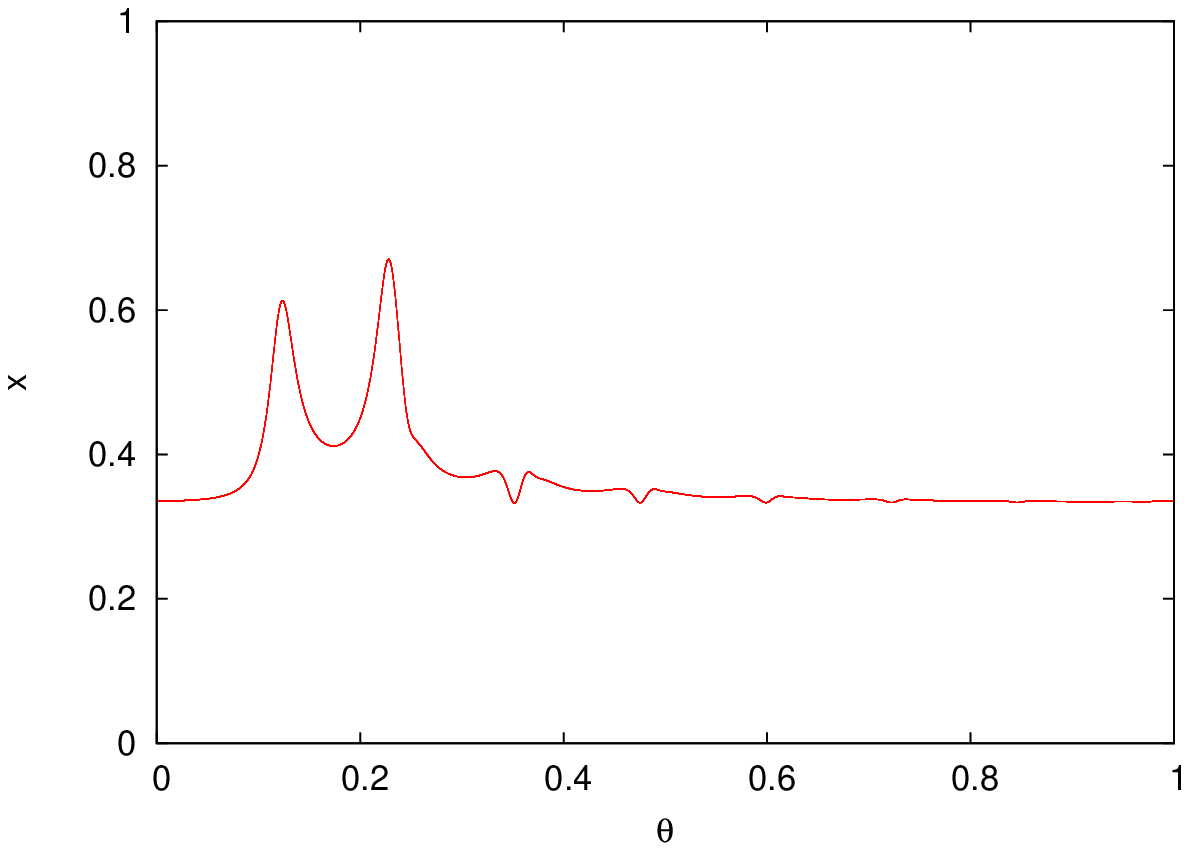}
    	\caption{The attractor when $\K = 0.5$.}\label{SNAGraph0_5}
    \end{subfigure}
\end{figure}

The concept of torus-collision has also been seen to cause loss of normal hyperbolicity in normally hyperbolic systems (see \cite{HaroLLaveHypBreakdown, BjerkHyperbolicity}).
In \cite{BjerkHyperbolicity} (where they study the projectivization of an invertible linear cocycle) the minimum distance between the tori were shown to vanish at linear speed with respect to the parameter.
It was remarked that this might be a universal phenomenon, occuring in a wide class of systems. Certainly, the same question could be asked about our model.

Returning to our model in (\ref{QPFLM}), we would like to understand the asymptotic process behind the degeneration of our smooth attractor into the SNA.
Our first result shows at which rate the minimum distance, from the repelling curve at $x = 0$ to the attractor, decreases, as $\K$ approaches 1.

In \cref{DistGraph}, we have plotted this minimum distance as obtained in our simulations.
The graph seems to suggest that the distance is asymptotically linear as $\K$ approaches 1 from below,
justifying similar observations in other models (\cite{HaroLLaveHypBreakdown, BjerkHyperbolicity}). We will prove that this is indeed the case.

\begin{figure}[h]
    \centering
    \includegraphics[scale=0.7]{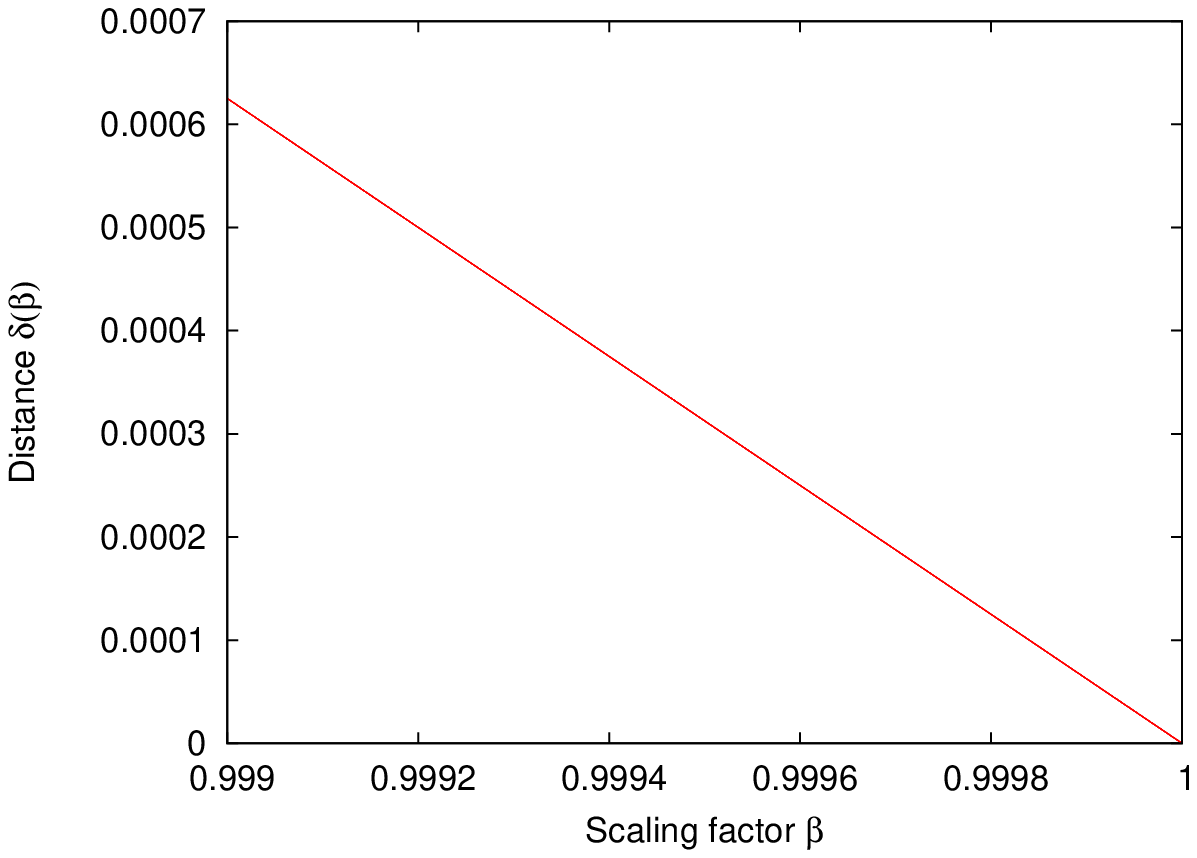}
    \caption{The minimum distance as a function of $\K$, when $\K$ is close to 1.}
    \label{DistGraph}
\end{figure}

Then, a more daring question presented itself: would it be at all possible to obtain asymptotics of how quickly the maximum derivative of the curve approaches infinity?
Our results yield the rather unexpected asymptotics that the derivative of the attractor, in the sup-norm, grows like
\begin{align*}
(1-\K)^{-1/2}
\end{align*}
as $\K$ approaches 1 from below, or approximately as one over the square root of the distance between our invariant curves.

The techniques used in this paper do not depend on the specific map, and we expect that similar systems exhibit the same asymptotic behaviours.
The exponent $-1/2$ does however seem to be related to the quadratic nature of our map,
more specifically to the non-vanishing of the second derivative of the attracting curve at the point closest to the repelling set.

It is also unknown what happens to our system (\ref{QPFLM}) when $\omega$ is not Diophantine.

\section{Model and results}

As in \cite{BjerkSNA}, let $\omega$ be an irrational number. We have introduced the parameter $0 \leq \K < 1$ to get the "extended" system (in the original model $\K = 1$ is fixed)
\begin{align*}
\Phi_{\alpha, \K}: \Torus \times [0,1] \to \Torus \times [0,1]: (\theta, x) \mapsto (\theta + \omega, c_{\alpha, \K}(\theta) \cdot p(x)),
\end{align*}
where
\begin{align*}
p(x) = x(1-x),
\end{align*}
is a quadratic (logistic) map, and
\begin{align*}
c_{\alpha, \K}(\theta) = \frac32 + \K\frac52\left( \frac1{1 + \lambda (\cos2\pi(\theta - \alpha/2) - \cos\pi\alpha)^2} \right),
\end{align*}
where $\lambda$ is assumed to be sufficiently large (depending on $\omega$), in order for the peaks to be narrow. The relationship between this $c(\theta)$ and the $a(\theta)$ is just that the $\frac32$ appearing there is moved into $c(\theta)$, and that we introduced one more parameter, $\alpha$.

The Diophantine condition reads
\begin{align*}
\inf \limits_{p \in \mathbb{Z}} |q\omega - p| > \frac{\kappa}{|q|^\tau} \text{ for all } q \in \mathbb{Z} \backslash \{0\} \tag*{$(DC)_{\kappa, \tau}$} \label{DC},
\end{align*}
for some $\kappa > 0, \tau \geq 1$. We note that the Diophantine irrationals have full (Lebesgue) measure on the interval $[0,1]$.

From this point on, we let $\omega$ be a fixed Diophantine irrational satisfying the condition \ref{DC} for some $\kappa > 0$ and $\tau \geq 1$.

For a given point $(\theta_0, x_0) \in \mathbb{T} \times [0,1]$, we write $(\theta_n, x_n) = \Phi^n(\theta_0, x_0)$.
The vertical Lyapunov exponent at the point $(\theta_0, x_0)$, we define as
\begin{align*}
\gamma(\theta_0, x_0) = \lim \limits_{n \to \infty} \frac1n \log \left| \frac{\partial x_n}{\partial x_0} \right| =
\lim \limits_{n \to \infty} \frac1n \sum \limits_{k = 0}^{n-1} \log |c(\theta_k)(1-2x_k)|,
\end{align*}
provided the limit exists. We define also
\begin{align*}
\overline{\gamma}(\theta_0, x_0) = \overline{\lim \limits_{n \to \infty}} \frac1n \sum \limits_{k = 0}^{n-1} \log |c(\theta_k)(1-2x_k)|.
\end{align*}

The following result is proved in \cite{BjerkSNA, BjerkSNA2}:

\begin{prop}\label{OldResult}
For all sufficiently large $\lambda > 0$, there is a parameter value $\alpha = \alpha_c$ such that the following holds for the map $\Phi = \Phi_{\alpha_c, \K = 1}$:
\begin{enumerate}[i)]
\item There is a strange attractor, the graph of a nowhere continuous measurable function $\psi: \mathbb{T} \to [0,1]$, which attracts points $(\theta, x)$, for a.e. $\theta \in \mathbb{T}$, and every $x \in (0,1)$.
\item $\overline{\gamma}(\theta, x) \leq \frac12 \log(3/5) < 0$ for a.e $\theta \in \mathbb{T}$ and every $x \in (0,1)$.
\item The attractor is dense in a 2D surface bounded by two continuous graphs, one identically 0, and the other one $h: \mathbb{T} \to [1/3, 1]$.
\end{enumerate}
\end{prop}

We are now ready to state the main theorem of this paper.

\begin{mainthm}
For all sufficiently large $\lambda > 0$, the following holds for the map $\Phi_\K = \Phi_{\alpha_c, \K}$, where $\alpha_c$ is as in \cref{OldResult}:
\begin{enumerate}[i)]
\item When $0 \leq \K < 1$, there is a curve, the graph of a $C^\infty$ function $\psi^\K: \mathbb{T} \to [0,1]$, which attracts every point $(\theta, x) \in \mathbb{T} \times (0,1)$.
\item When $0 \leq \K < 1$, $\overline{\gamma}(\theta, x) \leq \frac12 \log(3/5) < 0$ for every $\theta \in \mathbb{T}$ and every $x \in (0,1)$.
\item The (minimum) distance $\delta(\K)$ between the attractor $\psi^\K$ and the repelling set $\mathbb{T} \times \{0\}$, is asymptotically linear in $\K$, as $\K \to 1^-$, specifically
\begin{align*}
\delta(\K) = const \cdot (1 - \K) + o(1-\K),
\end{align*}
where the constant equals $c_{\K = 1}(\alpha_c + \omega) \cdot \frac58$.
\item The sup-norm of the derivative of $\psi^\K$ satisfies the asymptotic
\begin{align*}
\frac{C_1}{(1-\K)^{1/2}} \leq \|\partial_\theta \psi^\K\| \leq \frac{C_2}{(1-\K)^{1/2}},
\end{align*}
where $0 < C_1 \leq C_2$ are constants, as $\K \to 1^-$.
\end{enumerate}
\end{mainthm}

The above statements correspond to \cref{NegativeLyapunov} and \cref{SmoothnessOfAttractor,MinimumDistance,DerivativeGrowth}.

\begin{remark}
The existence of a smooth attractor is actually true for any $\alpha$, when $0 \leq \K < 1$, which can be shown using the techniques in this paper.
The truly "difficult" and interesting case is when $\alpha = \alpha_c$ (actually, by symmetry of the peaks, there should be a "mirror image" of $\alpha_c$ where there's an SNA).
Whenever $\alpha$ is not equal to $\alpha_c$ or its "mirror image", we expect there to be no "SNA", even when $\K = 1$ (thus postponing the bifurcation).
\end{remark}

\begin{remark}
The assumption that $\omega$ is Diophantine is for technical reasons (see \cref{DiophReturnTime}), to ensure that the orbits spend long periods away from certain "bad" regions.
We don't know if the results can be extended to non-Diophantine irrationals.
\end{remark}

Below, we will give a short discussion of the driving mechanism in our model responsible for the appearance of a smooth attractor, and later it's bifurcation into an SNA.
As long as $c(\theta)$ is close to $\frac32$ (such as when $\K$ is small), there will be an attractor given by the graph $(\theta, \psi(\theta))$ of some smooth function $\psi(\theta): \mathbb{T} \to [0,1]$ which is approximately $\frac13$.

The set $\mathbb{T} \times \{0\}$ is an invariant repelling set. The important feature of our model is that $x = 1$ is mapped directly to $x = 0$. Our $c_\alpha$ was made to be $c_\alpha(\theta) \approx \frac32$, except when $\theta$ is \textit{very} close to 0 and $\alpha$.

The interesting values of $\alpha$ will be close to $\omega$, in order to produce an orbit going through
\begin{align*}
(\alpha_c - \omega, \approx \frac13) \mapsto (\alpha_c, \frac12) \mapsto (\alpha_c + \omega, 1) \mapsto (\alpha_c + 2\omega, 0),
\end{align*}
culminating in a torus collision. This chain occurs when $\alpha = \alpha_c$ (the critical value in \cite{BjerkSNA}) and $\K = 1$. That is exactly when an SNA appears in our system.

This article has been divided into several sections, each with a separate goal in mind.

In section 3, we have collected several numerical lemmas for computations that are used repeatedly throughout the following sections.

Section 4 contains the big induction step, where we show that, excluding certain (possibly) degenerate sets, we have good control on expansion/contraction.
There, we also derive results which will be used to show that the induction can go on, even past these "degenerate sets".

All the results are tied together in section 5, which has been split into three separate parts.
In the first part, we show that there is a unique attracting curve which is the graph of a smooth map.
The second part deals with the minimum distance between the attractor and the repelling set $\mathbb{T} \times \{0\}$, and how this behaves asymptotically as the parameter $\K \to 1^-$.
Finally, in the third part, we will prove the bounds on the growth of the maximum derivative of the attracting curve.

At the beginning of each (sub)section, we will briefly sketch the main ideas of that section.
\section{Some preparations and lemmas for later}

Here, we will list some "numerical" (or "computational") lemmas to be used in the later sections.

The reason for choosing a Diophantine $\omega$ is that we then get a lower bound on the number of iterations required by the map $\theta \mapsto \theta + \omega$ to return to a small interval of $\mathbb{T}$ (\cref{DiophReturnTime}). This is a very important assumption used in our techniques.

\begin{lemma}\label{DiophReturnTime}
If $\omega \in \mathbb{T}$ satisfies the Diophantine condition \ref{DC}, and $I \subset \mathbb{T}$ is an interval of length $\varepsilon > 0$, then
\begin{align*}
I \cap \bigcup \limits_{0 < |m| \leq N} (I + m\omega) = \emptyset
\end{align*}
with $N = [(\kappa/\varepsilon)^{1/\tau}]$\footnote{$[x]$ denotes the integer part of $x$.}.
\end{lemma}

We will fix, for the remainder of this paper, the following notation.
\begin{align*}
\Phi_{\alpha, \K}: \Torus \times [0,1] \to \Torus \times [0,1]: (\theta, x) \mapsto (\theta + \omega, c_{\alpha, \K}(\theta) \cdot p(x)),
\end{align*}
where $\K \in [0,1]$, $\omega$ is a Diophantine irrational number,
\begin{align*}
p(x) = x(1-x)
\end{align*}
is the quadratic map, and
\begin{align*}
c_{\alpha, \K}(\theta) = \frac32 + \K\frac52\left( \frac1{1 + \lambda g(\theta, \alpha)^2} \right),
\end{align*}
where
\begin{align*}
g(\theta, \alpha) = \cos2\pi(\theta - \alpha/2) - \cos\pi\alpha.
\end{align*}
The constant $\lambda$ will be assumed sufficiently large throughout this paper. We will often suppress the parameters $\alpha, \K$ in our notation whenever they can be understood from context.

Given $(\theta_0, x_0)$, we will use the notation
\begin{align*}
(\theta_n, x_n) = \Phi^n(\theta_0, x_0), \quad n \geq 0.
\end{align*}

We will introduce a few intervals and constants of importance later in the induction. We let
\begin{align}
I_0 &= [-\lambda^{-1/7}, \lambda^{-1/7}]; \\
\mathcal{A}_0 &= [\omega - \lambda^{-2/5}/2, \omega - 2\lambda^{-2/3}].\label{DefinitionOfA0}
\end{align}
The interval $I_0$ contains most of the $\theta$ where $c$ has its first peak, and is the first zooming interval in the induction. The interval $\mathcal{A}_0$ is where some of the interesting values of $\alpha$ lie. In particular $\alpha_c \in \mathcal{A}_0$. There is one more such interesting interval, situated slightly to the right of $\omega$, but to keep derivatives positive, we have chosen to focus on the left side of the peak at 0. Needless to say, the same techniques apply to the other interval, except that some constants might have to be tweaked.

The constants are
\begin{align*}
M_0 &= [\lambda^{1/(14\tau)}]; \\
K_0 &= [\lambda^{1/(28\tau)}],
\end{align*}
where $[x]$ denotes the integer part of $x$. They have been chosen to be $M_0 \approx \sqrt{N}$, and $K_0 \approx N^{1/4}$, where $N$ is the minimal return time to $I_0$ in \cref{DiophReturnTime}.

Also, given an interval $I$, and a $\theta_0 \in \mathbb{T}$, we denote by $N(\theta_0; I)$ the smallest non-negative integer $N$ such that $\theta_N = \theta_0 + N\omega \in I$. Note that $N(\theta_0; I) = 0$ if $\theta_0 \in I$.

The "contracting" region $C$ is given by
\begin{align*}
C = [1/3 - 1/100, 1/3 + 1/100],
\end{align*}
and corresponds to the values of $x$ where there is strong contraction, as long as $\theta \not\in I_0 \cup (I_0 + \omega)$. This is the desirable place to be, and the whole induction step is devoted to showing that orbits spend almost all their time in this region.

The following lemmas will ascertain that the perturbations of the constant in the quadratic map $c(\theta)p(x)$ will be small when $\theta \not\in I_0 \cup (I_0 + \omega)$.

In the remainder of this section, whenever the proof of a statement is omitted, it can be found in \cite{BjerkSNA}.
For each lemma, we have indicated, in brackets, the corresponding one in \cite{BjerkSNA}.

\begin{lemma}[{\cite[Lemma 3.1]{BjerkSNA}}] \label{BoundsOnFunctionC}\label{WillGetOneHalf}\label{AlphaPeakZoom}\label{CDerivativeOutsideI0}\label{CDerivativeAroundSecondPeak}
For all sufficiently large $\lambda > 0$ the following hold for $\alpha \in \mathcal{A}_0$ and $0 \leq \K \leq 1$:
\begin{enumerate}[a)]
\item $|c_{\alpha, \K}(\theta) - \frac32|, |\partial_\theta c_{\alpha, \K}(\theta)|, |\partial_\K c_{\alpha, \K}(\theta)| < 1/\sqrt{\lambda}$ for every $\theta \not\in I_0 \cup (I_0 + \omega)$.
\item For any $0 \leq \delta \leq 1$, $\{ \theta : c(\theta) \geq \left(\frac32 + \K\frac52\right) \left(1 - \delta\right) \} \cap (I_0 + \omega) \subseteq [\alpha - \sqrt{\delta}\lambda^{-1/4}, \alpha + \sqrt{\delta}\lambda^{-1/4}]$.
\item For $0 \leq \K \leq 1, \alpha \in \mathcal{A}_0$ and $\theta \in I_0 + \omega$, $\K \lambda^{1/6} \leq \partial_\theta c_{\alpha, \K}(\theta) \leq \K \lambda$.
\end{enumerate}
\end{lemma}

\begin{proof}
For the second statement, we calculate the Taylor series at $\theta = \alpha$, to obtain

\begin{align*}
c(\theta) &= \frac32 + \K \frac52 - 10\K\lambda\pi^2\sin^2(\pi\alpha)(\theta - \alpha)^2 + \K\lambda O((\theta - \alpha)^3)
\end{align*}
Therefore,
\begin{align*}
c(\theta) \geq \left(\frac32 + \K\frac52 \right) \left(1 - \delta\right)
\end{align*}
implies that
\begin{align*}
\K\lambda \left( 10\pi^2\sin^2(\pi\alpha)(\theta - \alpha)^2 + O((\theta - \alpha)^3) \right) \leq \left(\frac32 + \K\frac52 \right) \delta
\end{align*}
Now, $c(\alpha \pm \sqrt{\delta}\lambda^{-1/4}) < \left(\frac32 + \K\frac52 \right) \left(1 - \delta\right)$, since
\begin{align*}
\K\lambda \left( 10\pi^2\sin^2(\pi\alpha)\delta\lambda^{-1/2} + O(\delta^{3/2}\lambda^{-3/4}) \right) &= \left( 10\pi^2\sin^2(\pi\alpha)\K\lambda^{1/2} + \cdot \K O(\delta^{1/2}\lambda^{1/4}) \right) \delta \\
&> \left(\frac32 + \K\frac52 \right) \delta
\end{align*}
when $\lambda > 0$ is large (independent of $\delta$). Since $c$ is smaller further away from the peak at $\alpha$, we are done.

The third statement is proved in \cite[Lemma 3.1]{BjerkSNA} for $\K = 1$. From this it immediately follows that
\begin{align*}
\K \lambda^{1/6} < \partial_\theta c_{\alpha, \K}(\theta) < \K \lambda
\end{align*}
for every $\alpha \in \mathcal{A}_0, \theta \in I_0 + \omega$, since $\partial_\theta c_{\alpha, \K}(\theta)$ is linear in $\K$.
\end{proof}

\begin{lemma}[{\cite[Lemma 3.2]{BjerkSNA}}]\label{CContractionLemma}
Provided that $\lambda > 0$ is sufficiently large, the following statements hold for $\alpha \in \mathcal{A}_0$ and $0 \leq \K \leq 1$:
\begin{itemize}
\item If $\theta_0 \not\in I_0 \cup (I_0 + \omega)$, and $x_0 \in C$, then $x_1 \in C$, and $|c(\theta_0)p'(x_0)| < 3/5$.
\medskip
\item If $\theta_0, \dots, \theta_{19} \not\in I_0 \cup (I_0 + \omega)$, and $x_0 \in [1/100, 99/100]$, then $x_{20} \in C$.\label{20IterationsToC}
\medskip
\item If $\theta_0 \not\in I_0 \cup (I_0 + \omega)$ and $x_0 \in [1/100, 99/100]$, then $x_1 \in (1/100, 2/5)$.\label{CloseToCIfAwayFromPeak}
\medskip
\item If $x_0 \in [0, 1/10]$, then $x_1 \geq \frac54 x_0$, for every $\theta_0 \in \mathbb{T}$.\label{AscentFromBottom}
\end{itemize}
\end{lemma}

\begin{lemma}[{\cite[Lemma 3.3]{BjerkSNA}}]\label{TwoStepsAfterEntry}
Suppose that $0 \leq \K \leq 1$. Then, if $\theta_0 \in \mathbb{T}$, $x_0 \geq 1/100$, and if $x_{-1} \in (0, 1/100) \cup (99/100, 1)$, then $x_2 \in [1/100, 99/100]$.
\end{lemma}

\begin{lemma}\label{IfInCThenFirstPeakDoesLittle}
Suppose that $x_0 \in C$, and $\theta_0 \in I_0$. Then for any $0 \leq \K \leq 1$
\begin{align*}
\frac3{10} < x_1 < \frac{99}{100},
\end{align*}
and
\begin{align*}
\frac1{100} < x_2.
\end{align*}
\end{lemma}

\begin{proof}
The assumption means that
\begin{align*}
1/3 - 1/100 \leq \psi^\K(\theta) \leq 1/3 + 1/100.
\end{align*}
Recall that
\begin{align*}
x_1 = c_{\alpha, \K}(\theta)p(x_0).
\end{align*}
This gives us the following bounds
\begin{align*}
\frac{3}{10} < \frac32 \cdot p(1/3 - 1/100) \leq x_1 \leq 4p(1/3 + 1/100) < 99/100,
\end{align*}
and therefore
\begin{align*}
x_1 \geq \frac32 p(99/100) > 1/100.
\end{align*}
\end{proof}

\begin{lemma}\label{TimeOfAscent}
Suppose that $x_0 < 1/100$. Then the smallest $T > 0$ satisfying that
\begin{align*}
x_T \geq 1/100,
\end{align*}
satisfies
\begin{align*}
T \leq \log_{5/4} \frac1{20x_0}.
\end{align*}
\end{lemma}

\begin{proof}
First, note that, since $c(\theta) \leq 4$, also $x_T \leq 4/100 = 1/20$, because otherwise $1/100 \leq x_{T-1}$.

Since $x_k < 1/100$ for every $0 \leq k < T$, using \cref{AscentFromBottom}, we get that
\begin{align*}
\left(\frac54\right)^T x_0 \leq x_T \leq \frac1{20},
\end{align*}
or
\begin{align*}
T \leq \log_{5/4} \frac1{20x_0}.
\end{align*}
\end{proof}

Applying the product rule and the chain rule, we obtain

\begin{align*}
\partial x_{n+1} = \left( \partial c(\theta_n) \right) \cdot p(x_n) + c(\theta_n) \cdot p'(x_n) \cdot \partial x_n,
\end{align*}
where $\partial$ denotes partial differentiation with respect to either $\theta$ or $\K$. We find inductively that
\begin{align}
\partial x_{n+1} = \left( \partial c(\theta_n) \right) \cdot p(x_n) + \partial x_0 \prod \limits_{j = 0}^n c(\theta_j) \cdot p'(x_j)
+ \sum \limits_{k = 1}^n \left( \partial \theta_{k-1} p(x_{k-1}) \prod \limits_{j = k}^n c(\theta_j) \cdot p'(x_j) \right).\label{IteratedPartialDerivatives}
\end{align}
Such products will be important to us, and we will control them by controlling products of the form $\prod \limits_{j = 0}^n |c(\theta_j) \cdot p'(x_j)|$.

The following lemma is an adaptation of \cite[Lemma 3.5]{BjerkSNA}.
\begin{lemma}\label{DerivativeBounds}
Assume that $x_0 \in [0,1]$, $\partial_\theta x_0 = \partial_\K x_0 = 0$, and $\prod \limits_{j=k}^T |c(\theta_j) p'(x_j)| < (3/5)^{(T - k + 1)/2}$ for every $k \in [0, T]$,
where $T > 10 \log \lambda$ is an integer. Assume moreover that $|\partial_\theta c(\theta_k)|, |\partial_\K c(\theta_k)| < 1/\sqrt{\lambda}$ for $k \in [T - 10 \log \lambda, T]$.
Then $|\partial_\theta x_{T + 1}|,|\partial_\K x_{T + 1}| < \lambda^{-1/4}$ provided that $\lambda$ is sufficiently large.
\end{lemma}

\begin{proof}
Exactly as in the proof of \cite[Lemma 3.5]{BjerkSNA}.
\end{proof}

The following lemma is a restatement of \cite[Lemma 3.4]{BjerkSNA} to include the parameter $\K$, and is used in the proof of the main theorem to give a lower bound on how long it takes $x_0$ to return to $C$ after having come really close to the peaks in the $\theta$-direction.

\begin{lemma} \label{GoodReturnBound}
Let $\alpha \in \mathcal{A}_0$, and $\K \in [0,1]$ be fixed. Set
\begin{align*}
J_M = \{ \theta : c(\theta, \alpha) \geq \left(\frac32 + \K\frac52\right) \left(1 - (4/5)^M\right) \} \cap (I_0 + \omega).
\end{align*}
Then, For all sufficiently large $\lambda > 0$, the following hold for $M \geq 10$: \\
Given $\theta_0 \in (I_0 - \omega) \backslash (J_M - 2 \omega)$, and $x_0 \in [\frac1{100}, \frac{99}{100}]$, there is a $3 \leq k \leq M - 7$ such that $x_k \in [\frac1{100}, \frac{99}{100}]$. \\
Given $\theta_0 \in I_0 \backslash (J_M - \omega)$, and $x_0 \in [1/100, 2/5]$, there is a $2 \leq k \leq M - 7$ such that $x_k \in [\frac1{100}, \frac{99}{100}]$. \\
Given $\theta_0 \in (I_0 + \omega) \backslash J_M$, and $x_0 \in [\frac1{100}, \frac{99}{100}]$, there is a $1 \leq k \leq M - 7$ such that $x_k \in [\frac1{100}, \frac{99}{100}]$. \\
\end{lemma}

The return time to the "good" region $[1/100, 99/100]$ is bounded by $M - 7$ regardless of the value of $\K$.

\begin{proof}
Exactly as the proof in \cite[Lemma 3.4]{BjerkSNA} (we may even use the exact same estimates).
\end{proof}

The following lemma is a complement to the one above, considering what happens when we reach the peak.
Now the behaviour is crucially dependent on the value of $\K$. The failure of such a statement when $\K = 1$ is what causes the SNA.
Keep in mind that $c(\theta) \leq (\frac32 + \K\frac52)$ for every $\theta$, and hence $c(\theta) < 4$ when $\K < 1$.

\begin{lemma} \label{BadReturnBound}
For all sufficiently large $\lambda > 0$, we have the following lemma.
Let $\alpha \in \mathcal{A}_0$, and $\K \in [0,1)$ be fixed. Set
\begin{align*}
J_M = \{ \theta : c(\theta, \alpha) \geq \left(\frac32 + \K\frac52\right) \left(1 - (4/5)^M\right) \} \cap (I_0 + \omega).
\end{align*}
Then, assuming that $M \geq 10$, there is a constant (integer) $M_C = M_C(\K)$, depending only on $\K$, such that: \\
Given $\theta_0 \in (J_M - 2 \omega) \subset I_0 - \omega$, and $x_0 \in [\frac1{100}, \frac{99}{100}]$, there is a $3 \leq k \leq M_C$ such that $x_k \in [\frac1{100}, \frac{99}{100}]$. \\
Given $\theta_0 \in I_0$, and $x_0 \in [1/100, 99/100]$, there is a $2 \leq k \leq M_C$ such that $x_k \in [\frac1{100}, \frac{99}{100}]$. \\
Given $\theta_0 \in J_M \subset I_0 + \omega$, and $x_0 \in [\frac1{100}, \frac{99}{100}]$, there is a $1 \leq k \leq M_C$ such that $x_k \in [\frac1{100}, \frac{99}{100}]$.
\end{lemma}

\begin{proof}
One satisfying, but not necessarily the smallest possible, value of $M_C$ is the following:
\begin{align*}
M_C = \frac{\log \frac1{150V_\K(1 - V_\K)}}{\log\frac54} + 4,
\end{align*}
where $V_\K = \frac38 + \K\frac58$. At the end of the proof, we will show that this constant is sufficient. \\

Suppose that $\theta_0 \in (J_M - 2 \omega)$, and $x_0 \in [1/100, 99/100]$.
Then by \cref{CloseToCIfAwayFromPeak}, $1/100 < x_1 < 2/5$, or $x_1 \in [1/100, 99/100]$.
Now, $1/100 < x_2 < (\frac32 + \K\frac52) p(1/2) \leq (\frac38 + \K\frac58) = V_\K$, regardless of $\theta_1 \in I_0$.
Since it is independent of $\theta_1$, the same proof as we do in $J_M - \omega$ will work for $\theta_1 \in I_0$.

In particular, if $x_2 \leq 99/100$, we only have to prove the last statement.
If however $99/100 < x_2 < V_\K$, the exact same argument as we will use to prove that case can be used.

Therefore, assume $x_2 \in [1/100, 99/100]$. The next iterate satisfies
\begin{align*}
1/100 \leq x_3 \leq (\frac32 + \K\frac52) p(1/2) \leq (\frac38 + \K\frac58) = V_\K.
\end{align*}
Since $\theta_3 \not \in I_0 \cup (I_0 + \omega)$, we obtain
\begin{align*}
\frac32 V_\K(1 - V_\K) \leq x_4 \leq 2/5.
\end{align*}
If $x_k < 1/100$, for $k \geq 3$, then by induction and \cref{AscentFromBottom} we get
\begin{align*}
x_{k+1} \geq (5/4)x_k \geq (5/4)^{k-4} \frac32 V_\K(1 - V_\K).
\end{align*}
Thus, to get a lower bound on the constant needed, we solve
\begin{align*}
\frac1{100} \leq \left(\frac54\right)^{k-4} \frac32 V_\K(1 - V_\K),
\end{align*}
whose solution is
\begin{align*}
k \geq \frac{\log \frac1{150V_\K(1 - V_\K)}}{\log\frac54} + 4.
\end{align*}
That is, it is sufficient to set $M_C \geq \frac{\log \frac1{150V_\K(1 - V_\K)}}{\log\frac54} + 4$, for the proof to hold.
\end{proof}
\section{The Induction}\label{SectionInduction}

In this section, we will build progressively longer chains of iterations, discarding certain starting values $(\theta, x)$,
and stopping the process once we reach close enough to the peaks.
We will bootstrap an induction scheme to show that these chains can be continued, for appropriate starting values, and passing the peaks at a "permissible" distance.

In more technical language, we will construct a nested sequence of sets $\mathbb{T} \supset \Theta_{-1} \supset \Theta_0 \supset \cdots \supset \Theta_n \supset \cdots$ of permissible
starting values of $\theta$. Along with this sequence, we construct a sequence $I_0 \supset I_1 \supset \cdots \supset I_n \supset \cdots$, of intervals "zooming in" on the critical
part of the peak, which will be $\alpha_c - \omega$ (where $\alpha_c$ is as in \cref{OldResult}).
This value is the interesting part of the first peak since it will "bump" the orbits into a region around $1/2$ (where the maximum of the quadratic family is attained),
preparing it for the next peak at $\alpha_c$.

We will then iterate a starting point $(\theta_0, x_0) \in \Theta_{n-1} \times C$, until $\theta_k \in I_n$. For every $\K < 1$, there is a "suitable scale" $I_{n(\K)}$, at which
this process can be easily continued beyond the set $I_{n(\K)}$.

Essentially, this continuation seems to be crucially dependent on the fact that the return of the orbit to the region $\Theta_{n(\K)} \times C$ (contracting region) occurs much
sooner than the return to the set $I_{n(\K)}$ (where the orbits may enter the expansive region).

The main result in this section is \cref{ZoomInductionStep}, which will be used repeatedly to get all the estimates we will need later.

\subsection{Base case}
Recall the set $I_0$ we considered in the previous section. Here we will show that we have control on orbits as long as $\theta_k \not\in I_0 \cup (I_0 + \omega)$.
The inductive step then shows what happens inside $I_0 \cup (I_0 + \omega)$.

We have made some slight alterations to the original statement in \cite{BjerkSNA}, but the proof is essentially the same and depends on the estimates in the previous section,
valid as long as $\theta_0 \not\in I_0 \cup (I_0 + \omega)$.

\begin{prop}\label{ZoomInductionZero}
Let $\alpha \in \mathcal{A}_0$ be fixed. There is a $\lambda_1 > 0$ such that if $\lambda \geq \lambda_1$, then the following hold: \\

\begin{enumerate}[(i)$_0$]
\item\label{ContractionStep0} If $\K \in [0,1]$,  $x_0, y_0 \in C$, and $\theta_0 \not\in I_0 \cup (I_0 + \omega)$, then, letting $N = N(\theta_0; I_0)$,
and $\xi_i \in \{tx_i + (1-t)y_i : t \in [0,1]\}$ be an arbitrary point between $x_i$ and $y_i$, for every $i \in [0, N-1]$, the following hold:
\begin{align}
\prod \limits^{N-1}_{i = k} |c(\theta_i)p'(\xi_i)| < (3/5)^{N - k} \quad \text{ for all } k \in [0, N - 1]; \label{TailContractionBaseInd}\\
\prod \limits^{k-1}_{i = 0} |c(\theta_i)p'(\xi_i)| < (3/5)^{k} \quad \text{ for all } k \in [1, N]; \label{ContractionRateBaseInd}\\
x_k \in C \quad \text{ for all } k \in [0, N]; \quad \text{ and} \label{AlwaysInCStepZero}\\
|x_k - y_k| < (3/5)^{k}|x_0 - y_0|, \quad \text{ for all } k \in [1, N].
\end{align}
\item If $\K \in [0,1]$, and $x_0 \in [1/100, 99/100]$, and $\theta_0 \not\in I_0 \cup (I_0 + \omega)$, then
\begin{align}
x_k \in [1/100, 99/100] \quad \text{ for all } k \in [0, N].\label{WhenInGoodStep0}
\end{align}
\end{enumerate}

\end{prop}

\subsection{Inductive step}
The inductive step works by zooming in on intervals $I_n \subset I_0$, and showing that we have a good control on orbits as long as $\theta_k \not\in I_n \cup (I_n + \omega)$.
At some point we must ask ourselves what happens to orbits when they enter $I_n$.
This is highly dependent on $\alpha$ and $\K$, but the essence of our method is that as long as $\K < 1$, we can find a suitable $I_n$
such that we will be able to retain control even throughout the interval $I_n$, and for all time thereafter.

We will begin by introducing some notation. Suppose that we are given intervals $I_0, \dots, I_n$, and constants $K_0, \dots, K_n, M_0, \dots, M_n$. We then define the sets

\begin{align}
&\Theta_n = \mathbb{T} \backslash \bigcup \limits_{i = 0}^n \bigcup \limits_{m = -M_i}^{M_i} (I_i + m\omega), \ \Theta_{-1} = \mathbb{T} \backslash (I_0 \cup (I_0 + \omega)), \\
&G_n = \bigcup \limits_{i = 0}^n \bigcup \limits_{m = 0}^{3K_i} (I_i + m\omega), \ G_{-1} = \emptyset, \\
&B_n = \{ \K : M_C(\K) \leq 2K_n - 2\} \label{Bn},
\end{align}
where $M_C(\K)$ is the constant in \cref{BadReturnBound}.

We see that, for every $n \geq 0$, the following hold
\begin{align*}
&\Theta_{n} \subseteq \Theta_{n-1} \\
&G_{n-1} \subseteq G_n \\
&B_n \subseteq B_{n+1}, \quad \text{ and } \bigcup \limits_{n = 0}^\infty B_n = [0,1)
\end{align*}
The ideas behind the respective sets are:

\begin{itemize}
\item The set $\Theta_n$ consists of the points $\theta \in \mathbb{T}$ that are far away from each of the intervals $I_0, \dots, I_n$.
Starting with a $\theta_0 \in \Theta_n$ gives us some "breathing room" before we get close to the peaks.
\item The set $G_n$ consists of the points $\theta$ which have recently visited one of the intervals $I_i$, and are well on their way to recover (start contracting again).
If we hit the peak at $I_0$, but stay away from $I_{n+1}$, then we should be close to $C$, in the $x$-direction (and far away from the peaks in the $\theta$-direction),
when we exit $G_n$, giving us a very long time to contract.
\item The set $B_n$ is the set of $\K$ for which it is necessary only to zoom as far as to the $n$-th scale (the interval $I_n$) in order to obtain good estimates on the contraction,
for all time, even past the return of $\theta$ to that interval.
\end{itemize}

The below proposition is a modified version of the main induction in \cite{BjerkSNA}, and some of the constructions have also been slightly modified.
This is the place where the Diophantine condition is used.

\begin{prop}\label{ZoomInductionStep}
Let $\alpha \in \mathcal{A}_0$ be fixed. There is a $\lambda_1 > 0$ such that if $\lambda \geq \lambda_1$, then the following hold: \\

Suppose that for some $n \geq 0$, we have constructed closed intervals $I_0 \supset I_1 \supset \cdots \supset I_n$, and chosen integers $M_0 < M_1 < \cdots < M_n$ and $K_0 < K_1 < \cdots < K_n$, satisfying
\begin{align}
&|I_k| = (4/5)^{K_{k-1}}, \quad K_k \in [(5/4)^{K_{k-1}/(4\tau)}, 2(5/4)^{K_{k-1}/(4\tau)}], \quad \text{ for } k = 1, 2, \dots, n;\label{IndConstants} \\
&M_k \in [(5/4)^{K_{k-1}/(2\tau)}, 2(5/4)^{K_{k-1}/(2\tau)}], \quad \text{ for } k = 1, 2, \dots, n; \quad \text{ and} \\
&I_n \supseteq [\alpha - (4/5)^{K_n}, \alpha + (4/5)^{K_n}]\label{IndInterval}.
\end{align}

Assume furthermore that the following holds:
\begin{enumerate}[(i)$_n$]
\item If $\K \in [0,1]$,  $x_0, y_0 \in C$, and $\theta_0 \in \Theta_{n-1}$, then, letting $N = N(\theta_0; I_n)$, and $\xi_i \in \{tx_i + (1-t)y_i : t \in [0,1]\}$ be an arbitrary point between $x_i$ and $y_i$, for every $i \in [0, N-1]$, the following hold:
\begin{align}
\prod \limits^{N-1}_{i = k} |c(\theta_i)p'(\xi_i)| < (3/5)^{(1/2 + 1/2^{n+1})(N - k)} \quad \text{ for all } k \in [0, N - 1]; \label{TailContraction}\\
\prod \limits^{k-1}_{i = 0} |c(\theta_i)p'(\xi_i)| < (3/5)^{(1/2 + 1/2^{n+1})k} \quad \text{ for all } k \in [1, N]; \label{HeadContraction}\\
x_k \not\in C \quad \text{ for some } k \in [0, N] \Rightarrow \theta_k \in G_{n-1}; \quad \text{ and} \label{WhenNotInC}\\
|x_k - y_k| < (3/5)^{(1/2 + 1/2^{n+1})k}|x_0 - y_0|, \quad \text{ for all } k \in [1, N], \label{ZoomContractionRate}\\
\bigcup \limits_{k = 0}^{20}(I_n + (2K_n + k)\omega) \subseteq \Theta_{n-1}, I_n - M_n \omega \in \Theta_{n-1}. \label{ReturnToTheta}
\end{align}
\item  If $\K \in [0,1]$, $x_0 \in [1/100, 99/100]$, and $\theta_0 \not\in I_0 \cup (I_0 + \omega)$, then
\begin{align}
x_k \not \in [1/100, 99/100] \text{ and } k \in [0, N(\theta_0; I_n)] \Rightarrow \theta_k \in G_{n-1}.\label{IfBadThenInG}
\end{align}
\item If $\K \in [0,1]$, $x_0 \in C$, and $\theta_0 \not\in I_n$, then, letting $N = N(\theta_0; I_n)$
\begin{align}
x_N \in C.\label{InCWhenReturnToBad}
\end{align}
\end{enumerate}

Then there is a closed interval $I_{n+1} \subset I_n$, and integers $M_{n+1}, K_{n+1}$ satisfying (\ref{IndConstants} - \ref{IndInterval})$_{n+1}$ such that $(i-iii)_{n+1}$ hold. \\

Moreover, under the same assumptions, the following holds:
\begin{enumerate}[(i)$_n$]\label{QuickReturnFromWorstToGood}
\setcounter{enumi}{3}

\item If $\K \in B_n$, $x_0 \in [1/100, 99/100]$, $0 \leq k \leq n$, and $\theta_0 \in (I_k - \omega) \cup I_k \cup (I_k + \omega)$, but $\theta_0 \not\in (I_j - \omega) \cup I_j \cup (I_j + \omega)$ for $k < j \leq n$, then
\begin{align}
&\theta_{2K_k + i} \in \Theta_{k-1}, \quad \text{ for every } i \in [0, 20]; \quad \text{ and} \\
&x_{2K_k + 20} \in C.
\end{align}
\end{enumerate}
\end{prop}

\begin{proof}
\Cref{DiophReturnTime} gives minimal return times
\begin{align*}
\begin{cases}
[(\kappa(4/5)^{K_{k-1}})^{1/\tau}] := N_k & k \geq 1 \\
[(2\kappa\lambda^{1/7})^{1/\tau}] := N_0 & k = 0
\end{cases}
\end{align*}
$N_k$ to the respective intervals $I_k$. The constants $M_k, K_k$ have been chosen to be $M_k \approx \sqrt{N_k}, K_k \approx \sqrt{M_k}$. By choosing $\lambda$ sufficiently large, we see that $N_k \gg M_k \gg K_k$.

In particular, \cref{DiophReturnTime} implies that
\begin{align}
I_k \cap \bigcup \limits_{0 < |m| \leq 10M_k} (I_k + m\omega) = \emptyset,
\end{align}
for every $k = 0, 1, \dots, n$. Also, since $3K_i < M_i$,
\begin{align*}
\bigcup \limits_{m = 0}^{3K_i} (I_i + m\omega) \subset \bigcup \limits_{m = -M_i}^{M_i} (I_i + m\omega)
\end{align*}
for every $k = 0, 1, \dots, n$, implying that
\begin{align}
\Theta_n \cap G_n = \emptyset\label{ThetaInterGEmpty},
\end{align}
for $n \geq -1$. Moreover, since $I_n \subset I_k$ ($k = 0, 1, \dots, n-1$), and $(I_k - \omega) \cap \left( \bigcup \limits_{m = 0}^{3K_k} (I_k + m\omega) \right)$ for $k = 0, 1, \dots, n-1$, we get that
\begin{align}
(I_n - \omega) \cap G_{n} = \emptyset\label{OneBeforeBadIsGood}.
\end{align} \\
\textit{Constructing the interval $I_{n+1}$:} \\
Let
\begin{align*}
I_{n+1} = [\alpha - (4/5)^{K_n}/2, \alpha + (4/5)^{K_n}/2].
\end{align*}
We have the inclusion
\begin{align*}
J_{2K_n} = \{ \theta : c(\theta) \geq (\frac32 + \K\frac52)(1 - (4/5)^{2K_n}) \} \subseteq [\alpha - \frac{(4/5)^{K_n}}{\lambda^{1/4}}, \alpha + \frac{(4/5)^{K_n}}{\lambda^{1/4}}] \subseteq I_{n+1}.
\end{align*}

This means, in particular, that by \cref{GoodReturnBound}, as long as $\theta_k \not\in \bigcup \limits_{m = -1}^1 \left( I_{n+1} + m\omega \right)$, we have good control on the contraction. \\ \\
\textit{Choosing the constants $K_{n+1}$, and $M_{n+1}$:} \\
See \cite[Proposition 4.2]{BjerkSNA}, where it is also shown that they satisfy \eqref{ReturnToTheta}$_{n+1}$. \\ \\
\textit{Verifying $(iii)_{n+1}$:} \\
Let $0 \leq s_1 < s_2 < \cdots < s_r = N$ be the return times to $I_n$. If $s_1 = 0$, then by assumption, $x_{s_1} = x_0 \in C$.
If $s_1 > 0$, then the induction hypothesis implies that $x_{s_1} \in C$. If $r = 1$, then we are done.

Suppose instead that we have proved that, for some $1 \leq l < r$ we have $x_{s_l} \in C$.
Since $\theta_{s_l} \in I_n \backslash I_{n+1}$, applying \cref{GoodReturnBound}, we get a $3 \leq t \leq 2K_n - 7$ such that $x_{s_l + t} \in [1/100, 99/100]$.

In the case that $\theta_{s_l + t} \not\in I_0 \cup (I_0 + \omega)$, then by $(ii)_n$, $x_{s_l + t + k} \not\in [1/100, 99/100]$ implies that $\theta_{s_l + t + k} \in G_{n-1}$.
Since, by \cref{ReturnToTheta}, $\theta_{s_l + 2K_n + i} \in \Theta_{n-1}$ ($i = 0, 1, \dots, 20$), which by \cref{ThetaInterGEmpty} is disjoint from $G_{n-1}$, we see that
$x_{s_l + 2K_n} \in [1/100, 99/100]$, and therefore $x_{s_l + 2K_n + 20} \in C$ by \cref{20IterationsToC}.

However, in the case that $\theta_{s_l + t} \in I_0 \cup (I_0 + \omega)$, assume that this $t$ is the smallest such time.
Now, $x_{s_l + t - 1} \not\in [1/100, 99/100]$ by our assumption on $t$, and by \cref{TwoStepsAfterEntry}, $x_{s_l + t + 2} \in [1/100, 99/100]$.
Since $\theta_{s_l + t + 2} \not\in I_0 \cup (I_0 + \omega)$, we may proceed as in the above paragraph to obtain $x_{s_l + 2K_n + 20} \in C$.

In any case, we have $\theta_{s_l + 2K_n + 20} \not\in I_n$, and $x_{s_l + 2K_n + 20} \in C$, and so $(iii)_n$ applies again, to conclude that $x_{s_{l+1}} \in C$.
By induction, we obtain our conclusion. \\ \\
\textit{Verifying $(i)_{n+1}$} \\
We want to prove that, for $N = N(\theta_0; I_{n+1})$,
\begin{align}
\prod \limits^{N-1}_{i = k} |c(\theta_i)p'(x_i)| < (3/5)^{(1/2 + 1/2^{n+2})(N - k)} \quad \text{ for all } k \in [0, N - 1]; \label{IndUpperProd} \\
\prod \limits^{k-1}_{i = 0} |c(\theta_i)p'(x_i)| < (3/5)^{(1/2 + 1/2^{n+2})k} \quad \text{ for all } k \in [1, N]; \label{IndLowerProd} \\
x_k \not\in C \quad \text{ for some } k \in [0, N] \Rightarrow k \in G_{n-1}; \quad \text{ and} \label{IndWhenNotInC} \\
|x_k - y_k| \leq (3/5)^{(1/2 + 1/2^{n+1})k}|x_0 - y_0|, \quad \text{ for all } k \in [1, N]. \label{IndZoomContractionRate}
\end{align}
We will designate, by \eqref{IndUpperProd}[$T$]-\eqref{IndZoomContractionRate}[$T$], the corresponding statements with $N$ replaced by an integer $T > 0$.

Begin by dividing the interval $[0, N]$ into parts
\begin{align*}
0 < s_1 < s_2 < \cdots < s_r = N,
\end{align*}
where the $s_l$ are the times when $\theta_{s_l} \in I_n$ (and $\theta_k \not\in I_n$ for $k \neq s_i$ for any $i$, and $0 \leq k \leq N$).

By the induction hypothesis, \eqref{IndLowerProd}[$s_1$] holds. Hence, if $r = 1$, we are done.
Suppose instead that $r > 1$, and that \eqref{IndLowerProd}[$s_l$] holds for $k \in [1, s_l]$, where $1 \leq l < r$.

Arguing as in the verification of $(iii)_{n+1}$ above, $x_{s_l + 2K_n + 20} \in C$. We already know that $\theta_{s_l + 2K_n + 20} \in \Theta_{n-1}$. Hence

\begin{align}
\prod \limits^{k-1}_{i = s_l + 2K_n + 20} |c(\theta_i)p'(\xi_i)| < (3/5)^{(1/2 + 1/2^{n+1})(k - s_l + 2K_n + 20)}\label{ProductTail}
\end{align}

for $k \in [s_l + 2K_n + 20 + 1, s_{l+1}]$.
Since $|c(\theta)p'(x)| \leq 4 < (5/3)^3$ for every pair $(\theta, x)$, we obtain the following bounds, valid for $k \in [1, 2K_n + 20]$

\begin{align*}
\prod \limits^{s_l + k-1}_{i = s_l} |c(\theta_i)p'(\xi_i)| < (5/3)^{3k}.
\end{align*}

Hence, for $k \in [1, 2K_n + 20]$, we have

\begin{align*}
\prod \limits^{s_l + k-1}_{i = 0} |c(\theta_i)p'(\xi_i)| < (3/5)^{(1/2 + 1/2^{n+1})s_l} \cdot (5/3)^{3k} \leq (3/5)^{(1/2 + 1/2^{n+1})s_l - 3k}.
\end{align*}

If we can show that $(1/2 + 1/2^{n+1})s_l - 3k > (1/2 + 1/2^{n+2})(s_l + k)$, we obtain the inequality, for $k \in [s_l + 1, s_l + 2K_n + 20]$,
\begin{align}
\prod \limits^{k-1}_{i = 0} |c(\theta_i)p'(\xi_i)| < (3/5)^{(1/2 + 1/2^{n+2})(s_l + k)} \label{ProductHead}.
\end{align}
This inequality indeed holds, since $K_n \gg 8 \cdot 2^{n+2}$, for $\lambda$ large enough, and $s_l \geq N_n > K_n^2$, yielding

\begin{align*}
(1/2 + 1/2^{n+1})s_l - 3k - (1/2 + 1/2^{n+2})(s_l + k) > 1/2^{n+2}N_n - 4k \\
> 1/2^{n+2}K_n^2 - 8K_n - 160 = K_n(1/2^{n+2}K_n - 8) - 160 > 0.
\end{align*}

Combining \cref{ProductTail} and \cref{ProductHead}, we obtain, for $k \in [s_l + 1, s_{l+1}]$, that

\begin{align*}
\prod \limits^{k-1}_{i = 0} |c(\theta_i)p'(\xi_i)| < (3/5)^{(1/2 + 1/2^{n+2})k}.
\end{align*}
By induction, \eqref{IndLowerProd}[$N$] holds, as was to be shown. The statement \eqref{IndUpperProd}[$N$] is proved in a similar fashion (the details are in \cite{BjerkSNA}).
The proof of \eqref{WhenNotInC}$_{n+1}$ is contained in \cite{BjerkSNA}.
The verification of \eqref{IndZoomContractionRate}[$N$] is now a quick application of the mean value theorem.\\ \\
\textit{Verifying $(ii)_{n+1}$} \\
As above, we begin by dividing the interval $[0, N]$ into parts
\begin{align*}
0 < s_1 < s_2 < \cdots < s_r = N,
\end{align*}
where the $s_l$ are the times when $\theta_{s_l} \in I_n$.

By the induction hypothesis, the following holds:
\begin{align*}
x_k \not \in [1/100, 99/100] \text{ and } k \in [0, s_1] \Rightarrow \theta_k \in G_{n-1} \subset G_n.
\end{align*}
Suppose that for some $1 \leq l < r$, we have for every $k \in [1, s_l]$ that
\begin{align*}
x_k \not \in [1/100, 99/100] \Rightarrow \theta_k \in G_n.
\end{align*}
Since $(I_n - \omega) \cap G_n = \emptyset$, we see that $x_{s_l - 1} \in [1/100, 99/100]$, and so there is a $3 \leq k \leq 2K_n - 7$ such that $x_{s_l + k} \in [1/100, 99/100]$ by \cref{GoodReturnBound}. Arguing as in the proof of $(iii)_{n+1}$ below, we see that $\theta_{2K_n} \in \Theta_{n-1}$, and $x_{2K_n} \in [1/100, 99/100]$. Hence, by $(ii)_n$, we have
\begin{align*}
x_k \not \in [1/100, 99/100] \text{ and } k \in [2K_n, s_1] \Rightarrow \theta_k \in G_{n-1} \subset G_n.
\end{align*}
Of course, since $\theta_k \in G_n$ for $0 \leq k \leq 3K_n$, we see that
\begin{align*}
x_k \not \in [1/100, 99/100] \text{ and } k \in [0, s_1] \Rightarrow \theta_k \in G_{n-1} \subset G_n.
\end{align*}
By induction, $(ii)_{n+1}$ holds.\\ \\
\textit{Verifying $(iv)_n$:} \\
Suppose that $0 \leq k \leq n$. Since $K \in B_n$, \cref{GoodReturnBound} (if $k < n$) and \cref{BadReturnBound} (if $k = n$) imply that there is a $1 \leq t \leq 2K_k - 2$ such that
\begin{align*}
x_t \in [1/100, 99/100].
\end{align*}
Suppose that this $t$ is the smallest such number. If $\theta_t \not\in I_0 \cup (I_0 + \omega)$, invoking $(ii)_k$, and noting that $\theta_{2K_k + j} \in \Theta_{k-1}$ for $j \in [0, 20]$, and $\Theta_{k-1} \cap G_{k-1} = \emptyset$, we obtain that
\begin{align*}
x_{2K_n} \in [1/100, 99/100];
\end{align*}
using \cref{20IterationsToC}, we see that $x_{2K_k + 20} \in C$, and $\theta_{2K_k + 20} \in \Theta_{k-1}$. \\

If $\theta_t \in I_0 \cup (I_0 + \omega)$, then as in the proof of $(iii)_{n+1}$ above, by \cref{TwoStepsAfterEntry} implies that $x_{t + 2} \in [1/100, 99/100]$. Since $\theta_{t + 2} \not\in I_0 \cup (I_0 + \omega)$, we just refer to the argument in the above paragraph, and conclude that the statement $(iv)_n$ holds true.
\end{proof}

\begin{corr}
By \cref{ZoomInductionZero}, $(i-iii)_0$ hold, where $(iii)_0$ just corresponds to \cref{AlwaysInCStepZero}, and so by \cref{ZoomInductionStep} $(i-iv)_n$ hold for every $n \geq 0$.
\end{corr}

\begin{lemma}\label{LocalControlOnProducts}
Suppose that $0 \leq \K \leq 1$, $x_0 \in C$, $\theta_0 \in \mathbb{T}$, and $0 < N$ satisfies that $\theta_i \not\in I_{m+1}$ for $i = 0, \dots, N$, where $m \geq 0$.
Then the following holds for every $0 \leq j < n \leq N$
\begin{align*}
\prod \limits_{i=j}^{n-1} |c(\theta_i)p(x_i)| \leq 4^{4K_m} \cdot (3/5)^{\left(1 - \frac1{M_0} \right)(n-j)/2}.
\end{align*}
\end{lemma}

\begin{proof}
Now, we will assume that $0 \leq j < n$.
If $x_j \in C$, let $s$ be the smallest integer satisfying $j \leq s$, and $\theta_s \in I_0$, say $\theta_s \in I_{p} \backslash I_{p + 1}$, where $p \leq m$ by assumption on $n$.
Set $p_1 = p$, and $t_1 = s$.

If $x_j \not\in C$, then let $s$ be the largest integer satisfying that $s \leq j$, $x_{t_i} \in C$, and $\theta_s \in I_0$.
As before, suppose that $\theta_s \in I_{p} \backslash I_{p + 1}$. Set $p_0 = p$, and $t_0 = s$.
Let $t_1$ be the next return time to $I_{p_0}$, say $\theta_{t_1} \in I_{p_1} \backslash I_{p_1 + 1}$.

If there is a next return time, less than $n$, to $I_{p_1}$, call the smallest such time $t_2$.
Suppose that $\theta_{t_2} \in I_{p_2} \backslash I_{p_2 + 1}$, where $p_1 \leq p_2$ by assumption.
Continue this process to get minimum return times $0 < t_1 < t_2 < \cdots < t_r \leq n$ to their corresponding intervals $\theta_{t_i} \in I_{p_i} \backslash I_{p_i}$,
where $0 \leq p_1 \leq p_2 \leq \cdots \leq p_r \leq m$ is an increasing sequence.

Decomposing our product into smaller ones, we obtain
\begin{align*}
\prod \limits_{i=j}^{n-1} |c(\theta_i)p(x_i)| = \left( \prod \limits_{i=j}^{t_1 - 1} |c(\theta_i)p(x_i)| \right) \cdots
\left( \prod \limits_{i=t_l}^{t_{l+1} - 1} |c(\theta_i)p(x_i)| \right) \cdots \left( \prod \limits_{i=t_r}^{n - 1} |c(\theta_i)p(x_i)| \right).
\end{align*}
Write $\theta_0^{(i)} = \theta_{t_i}, x_0^{(i)} = x_{t_i}$. The intermediate products satisfy
\begin{align*}
\prod \limits_{i=t_l}^{t_{l+1} - 1} |c(\theta_i)p(x_i)| &= \left( \prod \limits_{i=t_l}^{t_l + 2K_{p_l} + 19} |c(\theta_i)p(x_i)| \right)
\cdot \left( \prod \limits_{i=t_l + 2K_{p_l} + 20}^{t_{l+1} - 1} |c(\theta_i)p(x_i)| \right) \leq\\
&\leq 4^{2K_{p_l} + 20} \cdot (3/5)^{(t_{l + 1} - (t_l + 2K_{p_l} + 20))/2} \leq (3/5)^{(t_{l + 1} - t_l - 8K_{p_l} - 80)/2} \leq\\
&\leq (3/5)^{\left(1 - \frac1{M_0} \right)(t_{l + 1} - t_l)/2},
\end{align*}
since $\theta^{(l)}_{2K_{p_l} + 20} \in \Theta_{p_l-1}, x^{(l)}_{2K_{p_l} + 20} \in C$, and $t_{l + 1} - t_l \geq N_{p_l} \gg M_{p_l} \cdot K_{p_l} \geq M_0 \cdot K_{p_l}$.
If there are no intermediate products, set them equal to 1. The last product has the upper bound
\begin{align*}
\prod \limits_{i=t_r}^{n - 1} |c(\theta_i)p(x_i)| &\leq 4^{2K_{p_r} + 20} \cdot 4^{K_{p_r} + 10} \cdot (3/5)^{\left(1 - \frac1{M_0} \right)(n - t_r)/2} \leq\\
&\leq 4^{4K_{p_r}} \cdot (3/5)^{\left(1 - \frac1{M_0} \right)(n - t_r)/2},
\end{align*}
where we noted that
\begin{align*}
4^{k/4} \cdot (3/5)^{\left(1 - \frac1{M_0} \right)k/2} \geq (5/3)^{k/2} \cdot (3/5)^{\left(1 - \frac1{M_0} \right)k/2} > 1,
\end{align*}
or
\begin{align*}
4^{K_{p_r} + 10} \cdot (3/5)^{\left(1 - \frac1{M_0} \right)(2K_{p_r} + 20)/2} > 1,
\end{align*}
and that the contracting factor is $(3/5)^{n - 2K_{p_r} - t_r}$ if $n \geq 2K_{p_r} + 20$.

The only product which needs special treatment is the first one, depending on whether $x_j \in C$ or not (the two cases in the first paragraph).
In the case where $x_j \in C$, and $j < t_1$, where $t_1$ is the first return to $I_0$, we obtain
\begin{align*}
\prod \limits_{i=j}^{t_1 - 1} |c(\theta_i)p(x_i)| \leq (3/5)^{(t_1 - j)/2}.
\end{align*}
In the case where $x_j \in C$, and $j = t_1$, this was already treated as an intermediate product, or the last one (depending on whether we returned to $I_0$ between $j$ and $n$).

This gives us that the total product satisfies
\begin{align}\label{ProductBound}
\prod \limits_{i=j}^{n-1} |c(\theta_i)p(x_i)| \leq 4^{4K_m} \cdot (3/5)^{\left(1 - \frac1{M_0} \right)(n - j)/2},
\end{align}
The last case to consider is the one where $x_j \not\in C$, and $t_0 \leq j$ satisfies $x_{t_0} \in C$.
Necessarily, $j < t_0 + 2K_{p_0} + 20$, since $x^{(0)}_{2K_{p_0} + 20} \in C$ (see \cref{QuickReturnFromWorstToGood}),
meaning that the next time $\theta^{(0)}_{2K_{p_0} + 20 + s} \in I_0$ (the smallest such integer $s > 0$), $x^{(0)}_{2K_{p_0} + 20 + s} \in C$.
Therefore, if $t_0 + 2K_{p_0} + 20 \leq j \leq t_0 + 2K_{p_0} + 20 + s$ this would contradict our assumption that $x_j \not\in C$;
whereas $t_0 + 2K_{p_0} + 20 + s < j$ would contradict our initial choice of $t_0$ (the last return to $I_0$, before $j$, such that $x_{t_0} \in C$).

Thus, setting $t = t_0 + 2K_{p_0} + 20$, the first product satisfies
\begin{align*}
\prod \limits_{i=j}^{t_1 - 1} |c(\theta_i)p(x_i)| = \left( \prod \limits_{i=j}^{t - 1} |c(\theta_i)p(x_i)| \right)
\cdot \left( \prod \limits_{i=t}^{t_1 - 1} |c(\theta_i)p(x_i)| \right) \leq 4^{4K_{p_0}} \cdot (3/5)^{\left(1 - \frac1{M_0} \right)(t_1 - j)/2} < 1,
\end{align*}
since $t_1 - t \gg N_{p_0} \gg 2K_{p_0} + 20$. It follows that we get the same bound on the product as in \eqref{ProductBound}.

For the case where $j = 0$, we note that $x_j \in C$, and therefore we have the upper bound 
\begin{align*}
\prod \limits_{i=0}^{2K_m + 20} |c(\theta_i)p(x_i)| \leq 4^{2K_m + 20} \leq 4^{4K_m} \cdot (3/5)^{\left(1 - \frac1{M_0} \right)(2K_m + 20)/2}.
\end{align*}
Taking into account the contraction, as we had analyzed the "constituent products" above, and using the above estimate for the maximum expansion,
we obtain the inequality
\begin{align*}
\prod \limits_{i=0}^{n-1} |c(\theta_i)p(x_i)| \leq 4^{4K_m} \cdot (3/5)^{\left(1 - \frac1{M_0} \right)n/2}.
\end{align*}
\end{proof}

\begin{lemma}\label{LocalControlOnSumsOfProducts}
Suppose that $0 \leq \K \leq 1$, $x_0 \in C$, and $0 < N$ satisfies that $\theta_i \not\in I_{m+1}$ for $i = 0, \dots, N$, where $m \geq 0$.
Then the following holds for every $0 \leq j < n \leq N$
\begin{align*}
\sum \limits_{j = 1} |\partial_\theta c(\theta_{j-1}) p(x_{j-1})| \prod \limits_{i=j}^{n-1} |c(\theta_i)p(x_i)| \leq 4^{4K_m} \cdot \frac{(3/5)^{\left(1 - \frac1{M_0} \right)(n-k)/2}}{1 - }.
\end{align*}
\end{lemma}

\section{Proof of the main theorem}
This section has been split into three parts covering existence and smoothness of attractor, minimum distance to repelling set, and growth of derivative, respectively.

We will use the same notation as in \cref{SectionInduction}.
Throughout this section we will assume that $\lambda$ is a fixed constant, and sufficiently large for every result in the previous sections to hold.
From now on, we will also assume that $\alpha = \alpha_c$. Note that $\alpha_c$ depends on $\lambda$.

A notation we will introduce in this section is $I_{n(\K)}$, where $0 \leq \K < 1$, and $n = n(\K)$ is the smallest integer satisfying $\K \in B_{n}$.

\subsection{Existence and regularity of the attractor}

Here we show that, for every $0 \leq \K < 1$, there is an attractor which is the graph of an invariant smooth ($C^\infty$) function $\psi^\K: \mathbb{T} \to (0,1)$,
and that this attractor depends smoothly on $\K$. This is the contents of \cref{SmoothnessOfAttractor}.

In order to accomplish this goal, we will follow a standard argument.
We will first show that there is an invariant space $S_n = \mathbb{T} \times B_n \times [\epsilon_n, 1 - \epsilon_n]$ for every $n \geq 0$,
such that for $(\theta, \K, x) \in S_n$, we have the uniform bound
\begin{align*}
\| \partial_x x_k \| \leq const \cdot \delta^k,
\end{align*}
for some $0 < \delta < 1$, where $\theta_0 = \theta, x_0 = x$.
This will give us a family, for every $n \geq 1$, $\{\psi_{\K, n}: \mathbb{T} \to (0,1)\}_{\K \in B_n}$, of smooth functions for, the graphs of which will be the (unique) attractor corresponding to that $\K$.
As we increase $n$, we will obtain a family $\{\psi^\K: \mathbb{T} \to (0,1)\}$ of smooth functions (attracting graphs) for every $0 \leq \K < 1$.

\begin{lemma}\label{EverythingAligns}
Assume that $\K \in B_n$ (in particular $0 \leq \K < 1$) for some $n \geq 0$.
If $\theta_0 \in \mathbb{T}$, and $x_0 \in (0,1)$, then there is a $0 \leq t$, such that $\theta_t \in \Theta_{n-1}$, and $x_t \in C$.

Moreover, if $x_0 \in (\epsilon, 1 - \epsilon)$, there is a $T_\epsilon \geq 0$ such that $t \leq T_\epsilon$.
In particular, if $\epsilon = 1/100$, we may choose $T_\epsilon \leq 2M_{n-1} + 1$.
\end{lemma}

\begin{proof}
Since $\frac32 \leq c(\theta) < 4$ for every $\theta \in \mathbb{T}$ when $0 \leq \K < 1$, it follows that $x_k \in (0,1)$ for every $k \geq 0$ ($0 < x_i < 4 p(\frac12) = 1$).

We will first show that there is an $s \geq 0$ such that $x_s \in [1/100, 99/100]$, and $\theta_s \not\in I_0 \cup (I_0 + \omega)$.
Then we will prove the statement from there.

Suppose first that $x_0 \in [1/100, 99/100]$. If $\theta_0 \not\in I_0 \cup (I_0 + \omega)$, we are done.

Assume instead that $\theta_0 \in I_0 \cup (I_0 + \omega)$. If $x_2 \in [1/100, 99/100]$, we are done.
Otherwise, $x_2 \not\in [1/100, 99/100]$, and we fall into one of the cases considered below.

Now, suppose instead that $x_0 \not\in [1/100, 99/100]$. Then there is an $s > 0$ such that $x_s \in [1/100, 99/100]$. Let $s$ be the smallest such integer.
Since $p(1 - x) = p(x)$, we may assume that $x_0 < 99/100$ (discounting the possibility that $x_0 > 99/100$.
By \cref{TimeOfAscent}, there is a uniform upper bound on $s$, say $s \leq S_\epsilon$, if $x_0 \in (\epsilon, 1 - \epsilon)$.

If $\theta_s \not\in I_0 \cup (I_0 + \omega)$, we are done. If instead $\theta_s \in I_0 \cup (I_0 + \omega)$, then since $s$ was the smallest such integer,
$x_{s - 1} \not\in [1/100, 99/100]$, and so by \cref{TwoStepsAfterEntry}, $x_{s + 2} \in [1/100, 99/100]$, and $\theta_{s+2} \not\in I_0 \cup (I_0 + \omega)$.

In any case, there is a (uniformly) bounded $s \leq S_\epsilon + 2$, such that $\theta_s \not\in I_0 \cup (I_0 + \omega), x_s \in [1/100, 99/100]$.

We may thus assume (without loss of generality) that $\theta_0 \not\in I_0 \cup (I_0 + \omega), x_0 \in [1/100, 99/100]$.
Recall that $\Theta_{n-1} \cap G_{n-1} = \empty$ by \cref{ThetaInterGEmpty}. Then \cref{WhenNotInC} implies that, the next time $t \geq 0$ that
$\theta_t \in \Theta_{n-1}$, then $x_t \in C$. Since $\Theta_{n-1} = \mathbb{T} \backslash \bigcup \limits_{i = 0}^{n-1} \bigcup \limits_{m = -M_i}^{M_i} (I_i + m\omega)$,
the maximum number of consecutive iterations spent outside $\Theta_{n-1}$ is $2M_{n-1} + 1$. Thus, setting $T_\epsilon = S_\epsilon + 2M_{n-1} + 3$, the proof is completed.
\end{proof}

\begin{lemma}\label{Contraction}
Let $n \geq 0$ be arbitrary. If $\K \in B_n$, $\theta_0 \in \Theta_{n-1}$, and $x_0, y_0 \in C$, then for each $k > 1$
\begin{align*}
|x_k - y_k| < (3/5)^{k/2} |x_0 - y_0|.
\end{align*}
\end{lemma}

\begin{proof}
Let $0 < s_1 < s_2 < \cdots$ be the times when $\theta_{s_l} \in I_n$. By \cref{ZoomContractionRate}
\begin{align*}
|x_k - y_k| < (3/5)^{(1/2 + 1/2^{n+1})k} |x_0 - y_0|,
\end{align*}
for $k \in [1, s_1]$. Since $s_1 \geq M_n \gg 20 \cdot 2^{n+1} K_n$ if $\lambda$ is large enough (as in \cref{ZoomInductionStep}), we obtain
\begin{align*}
|x_{s_1} - y_{s_1}| < (3/5)^{s_1/2 + 20K_n} |x_0 - y_0|.
\end{align*}

Suppose that $|x_{s_l} - y_{s_l}| < (3/5)^{s_l/2 + 20K_n} |x_0 - y_0|$ holds for $l \geq 1$.
Since $\K \in B_n$, $(iv)_n$ implies that $\theta_{s_l + 2K_n + 20} \in \Theta_{n-1}$, and $x_{s_l + 2K_n + 20} \in C$.
Recall that $|c(\theta)p'(x)| < 4 < (5/3)^3$ for every $\theta \in \mathbb{T}$ and $x \in [0,1]$. Now, it follows that
\begin{align*}
|x_{s_l + k} - y_{s_l + k}| < 4^k \cdot |x_{s_l} - y_{s_l}| < (5/3)^{3k} \cdot (3/5)^{s_l/2 + 20K_n} \cdot |x_0 - y_0|,
\end{align*}
for $k \in [1, 2K_n + 20]$. Since $k < 2K_n + 20$, and therefore $20K_n - 3k \geq 10K_n \geq k/2$, we get
\begin{align*}
|x_{s_l + k} - y_{s_l + k}| < (3/5)^{s_l/2 + 20K_n - 3k} \cdot |x_0 - y_0| < (3/5)^{s_l/2 + k/2} \cdot |x_0 - y_0|.
\end{align*}
Now, we obtain for $k \in [s_l + 2K_n + 20, s_{l+1}]$ that
\begin{align*}
|x_k - y_k| &< (3/5)^{(1/2 + 1/2^{n+1})(k - s_l + 2K_n + 20)} \cdot |x_{s_l + 2K_n + 20} - y_{s_l + 2K_n + 20}| <\\
&< (3/5)^{(1/2 + 1/2^{n+1})(k - s_l + 2K_n + 20)} \cdot (3/5)^{(s_l + 2K_n + 20)/2} \cdot |x_0 - y_0| =\\
&= (3/5)^{k/2 + 1/2^{n+1}(k - s_l + 2K_n + 20)} |x_0 - y_0|.
\end{align*}
We will now proceed to prove the stronger bound for $k = s_{l+1}$.
We know that $1/2^{n+1}(s_{l+1} - s_l) \geq 1/2^{n+1} N_m \gg 1/2^{n+1} (20 \cdot 2^{n+1} K_n$ (again, see the proof of $(i)_{n+1}$, \cref{ZoomInductionStep})
\begin{align*}
|x_{s_{l+1}} - y_{s_{l+1}}| &< (3/5)^{s_{l+1}/2 + 1/2^{n+1}(s_{l+1} - s_l + 2K_n + 20)} |x_0 - y_0| =\\
&< (3/5)^{s_{l+1}/2 + 20K_n} |x_0 - y_0|
\end{align*}
By induction, the statement follows.
\end{proof}

\begin{lemma}\label{InvariantSubset}
For every $n \geq 0$, there exists an invariant (compact) subset $S_n = \mathbb{T} \times B_n \times [a_n, 1 - a_n]$,
where $0 < a_n \leq 1/4$, such that for $(\theta_0, \K, x_0), (\theta_0, \K, y_0) \in S_n$
\begin{align*}
|x_k - y_k| < c_n \cdot (3/5)^{k/2} |x_0 - y_0|,
\end{align*}
where $c_n > 0$ is a constant depending only on $n$.
\end{lemma}

\begin{proof}
Suppose that $\K_{\max} < 1$ is the biggest $\K \in B_n$. Let
\begin{align*}
b_n = \max \limits_{\K \in B_n, \theta \in \mathbb{T}} c_\K(\theta)p(1/2) = 1/4 \cdot \left( 3/2 + 5/2\K_{\max} \right) < 1.
\end{align*}
We will show that $a_n = 1 - b_n$ will suffice. Let $\theta_0 \in \mathbb{T}, x_0 \in [a_n, 1 - a_n]$. Note that, if
$x_0 \not\in [1/100, 99/100]$, then, for every $\K \in B_n$,
\begin{align*}
\frac98 a_n \leq \frac32 a_n (1 - a_n) \leq c_\K(\theta_0)p(x_0) = x_1 \leq c_\K(\theta_0) a_n (1 - a_n) \leq 4 \cdot 1/4 \cdot (1 - a_n),
\end{align*}
since $1 - a_n \geq \frac34$. That is, $x_1 \in S_n$. Since this worked for any $\theta_0 \in \mathbb{T}$, this set must be invariant.

For the second part, let $\theta_0 \in \mathbb{T}, x_0, y_0 \in S_n$. According to \cref{EverythingAligns}, there are $s, t \leq T_n$,
such that $\theta_s, \theta_t \in \Theta_{n-1}, x_s, y_t \in C$, where $T_n$ is the same for all these starting values. We may assume without loss of generality
that $s \leq t$. Recall that $\Theta_{n-1} \cap G_{n-1} = \empty$ by \cref{ThetaInterGEmpty}. Since $\theta_s \in \Theta_{n-1}, x_s \in C \subset [1/100, 99/100]$, and $\theta_t \in \Theta_{n-1}$,
\cref{WhenNotInC} implies that $x_t \in C$. Hence $\theta_t \in \Theta_{n-1}$, and $x_t, y_t \in C$. Now,
\begin{align*}
|x_t - y_t| \leq 4^t \cdot |x_0 - y_0|.
\end{align*}
Combining this with \cref{Contraction} yields, for every $k \geq 0$,
\begin{align*}
|x_k - y_k| \leq 4^{T_n} \cdot (5/3)^{T_n/2} \cdot (3/5)^{k/2} |x_0 - y_0|,
\end{align*}
which concludes our proof.
\end{proof}

\begin{corr}\label{NegativeLyapunov}
For every $(\theta_0, \K, x_0) \in S_n$ ($n \geq 0$), and every for every $k > 0$,
\begin{align*}
\left| \frac{\partial x_k}{\partial x_0} \right| < c_n \cdot (3/5)^{k/2},
\end{align*}
for some constant $c_n$ depending only on $n$.
\end{corr}

\begin{proof}
Choose $x_0$ in the interior of $A_\K$. We have for small enough $|h| > 0$ that $x_0 + h, x_0 \in A_\K$. Considering $x_k(x_0)$ as a function of $x_0$, we have
\begin{align*}
\left| \frac{\partial x_k}{\partial x_0} \right| &= \left| \lim \limits_{h \to 0} \frac{x_k(x_0 + h) - x_k(x_0)}{h} \right| \\
&< \lim \limits_{h \to 0} \frac{c_n \cdot (3/5)^{k/2}|x_0 + h - x_0|}{|h|} \\
&= c_n \cdot (3/5)^{k/2}.
\end{align*}
\end{proof}

\begin{prop}\label{SmoothnessOfAttractor}
There is an invariant curve, the graph of a function $\psi^\K(\theta)$ which is smooth smooth ($C^\infty$) in both $\K$ and $\theta$.
This curve attracts the orbits of every point $(\theta, x) \in \mathbb{T} \times (0,1)$.
\end{prop}

\begin{proof}
We will use the results in \cite{StarkRegularityQPF}. In his notation, for a fixed $n \geq 0$, $(\theta, \K) \in X = \mathbb{T} \times B_n$ and $x \in Y = [a_n, 1 - a_n]$
(where $a_n$ is as in \cref{InvariantSubset}). Now, by \cref{NegativeLyapunov}
\begin{align*}
\left| D_x x_k \right| < c_n \cdot (3/5)^{k/4}
\end{align*}
for every $(\theta_0, \K, x_0) \in S_n = X \times Y$.

Applying \cite[Theorem 2.1]{StarkRegularityQPF}, we obtain continuous invariant graphs $\{\psi^\K_n: \mathbb{T} \to (0,1)\}$ for each $\K \in B_n$,
attracting all of $\mathbb{T} \times (0,1)$, by \cref{EverythingAligns}.

Now, \cite[Theorem 3.1]{StarkRegularityQPF} implies that each $\psi^\K_n$ is as smooth as $\Phi_{\alpha_c, \K}$, that is $C^\infty$.

If $\K \in B_n$, then $\K \in B_m$ and $\psi^\K_m = \psi^\K_n$ for every $m \geq n$, since the attractor is unique. We also recall that $\bigcup \limits_{n = 0}^\infty B_n = [0,1)$.
Therefore, we obtain for every $0 \leq \K < 1$ a $C^\infty$ map
\begin{align*}
\psi^\K: \mathbb{T} \to (0,1),
\end{align*}
the graph of which attracts $\mathbb{T} \times (0,1)$.
\end{proof}

\subsection{Asymptotic minimal distance between attractor and repeller}

Here, we show that, when $\K \in B_n$, then the curve $\psi^\K$ will be essentially flat in the step before the first peek,
i.e. that $\partial_\theta \psi^\K(I_n)$ is very small, and furthermore, it will be located in $C$.
This will then give us very good bounds on $\partial_\theta \psi^\K(I_n + \omega)$, which will be very close to $\partial_\theta c(I_n)$.
That is $\psi^\K(I_n + \omega)$ will almost look like $c$ does slightly to the left of the peak at $\theta = 0$, that is, sharply increasing.

The next part is to show that the value of $\psi^\K(\alpha_c)$ is almost $1/2$, meaning that $\psi^\K(\alpha_c + \omega) \approx c(\alpha_c)p(1/2)$ is close to the "potential maximum".
For $\theta \in I_n + \omega$ not very close to $\alpha_c$,
the sharp nature of the peak at $\alpha_c$ will mean that $\psi^\K(\theta + \omega)$ can't reach as high as $\psi^\K(\alpha_c + \omega)$.
This will then give us the asymptotic behaviour of the minimum distance we described.

The main results here are \cref{DerivativeBoundBeforePeak,MinimumDistance}.

\begin{lemma}\label{SmallIteratedDerivatives}
If $\theta_0 \in \Theta_{n-1}$, and $x_0 = x \in C$, then
\begin{align*}
\left|(\partial_\theta c(\theta_{N-1})) \cdot p(x_{N-1}) + \sum \limits_{j = 1}^{N-1} (\partial_\theta c(\theta_{j-1})) \cdot p(x_{j-1}) \prod \limits_{i=j}^{N-1} c(\theta_i) \cdot p'(x_i)\right| < \lambda^{1/4},
\end{align*}
where $N = N(\theta_0; I_n)$, and $\partial$ is either $\partial_\K$ or $\partial_\theta$.
\end{lemma}

\begin{proof}
Note that the assumption that $\partial_\theta x_0 = 0$, is equivalent to
\begin{align*}
|\partial_\theta x_N| = \left|(\partial_\theta c(\theta_{N-1})) \cdot p(x_{N-1}) + \sum \limits_{j = 1}^{N-1} (\partial_\theta c(\theta_{j-1})) \cdot p(x_{j-1}) \prod \limits_{i=j}^{N-1} c(\theta_i) \cdot p'(x_i)\right|,
\end{align*}
since then $\partial_\theta x_0 \prod \limits_{i=j}^{N-1} c(\theta_i) \cdot p'(x_i)$ is removed from the expression.

Let $s < N$ be the smallest integer such that $\theta_i \not\in I_0 \cup (I_0 + \omega)$ for $s \leq i \leq N$ (that is $\theta_i$ won't return to $I_0$ before $i = N$). Since
\begin{align*}
\theta_0 \in \Theta_{n-1} = \mathbb{T} \backslash \bigcup \limits_{i=0}^{n-1} \bigcup \limits_{k = -M_i}^{M_i} I_0 + k\omega,
\end{align*}
and also $N(\theta; I_0) \geq M_0$ for $\theta \in I_0$, we deduce that at least $s \geq M_0$.

Recall that $M_0 \gg K_0 = \lambda^{1/28}$, and so $K_0 \gg 10 \log \lambda$ if $\lambda$ is large.
Thus, for every $s \leq k \leq N$, $\theta_k \not\in I_0 \cup (I_0 + \omega)$, and $|\partial_\theta c(\theta_k)|, |\partial_\K c(\theta_k)| < \frac1{\sqrt{\lambda}}$ (see \cref{CDerivativeOutsideI0}),
and also $\prod \limits_{j = k}^{N-1} |c(\theta_j)p(x_j)| < (3/5)^{(N - k)/2}$ (see \cref{TailContractionBaseInd}).

Applying \cref{DerivativeBounds} for $T = N-1$, assuming $\partial_\theta x_0 = 0$, we obtain that $|\partial_\K x_N|, |\partial_\K y_N| \leq \lambda^{-1/4}$, which is what we wanted to show.
\end{proof}

Let $0 \leq \K < 1$ be fixed. For each given $(\theta_0, x_0) \in I_0 \times C$, set $T(\theta_0, x_0)$ equal to the smallest positive integer $T > 2$ such that
\begin{align*}
x_T \geq \frac1{100}.
\end{align*}
Set $T(\theta) = T(\theta, \psi^\K(\theta))$.

\begin{lemma}\label{DerivativeBoundBeforePeak}
Suppose that $0 \leq \K < 1$, and let $J = J(\K)$ be an interval such that
\begin{align*}
I_{m+1} \subseteq J \subseteq I_m,
\end{align*}
for some $1 \leq m$, satisfying that, for every $\theta_0 \in J$,
\begin{align*}
T(\theta_0) \leq (N_m)^{3/4},
\end{align*}
where
\begin{align*}
N_m = \min \limits_{\theta \in (I_m + \omega)} N(\theta; I_m).
\end{align*}
Then
\begin{align*}
|\partial_\theta \psi^\K(\theta)|, |\partial_\K \psi^\K(\theta)| \leq \lambda^{-1/4} + \epsilon(m)
\end{align*}
for every $\theta \in J$, where $\epsilon(m) \to 0$ as $m \to \infty$.
Moreover,
\begin{align}
\psi^\K(J) \subseteq C,\label{InCAtFirstPeak}
\end{align}
and if $m \geq 1$ is large enough,
\begin{align*}
\K \lambda^{1/7} \leq \partial_\theta \psi^\K(\theta) \leq \K \lambda
\end{align*}
for $\theta \in J + \omega$.
\end{lemma}

\begin{proof}
We will iterate the segment given by $\theta_0 = \theta \in J \subseteq I_0$. For ease of notation, we set $x_0 = \psi^\K(\theta_0)$.

Let $0 = s_0 < s_1 < \dots$ be the return times to $J$, that is for $i \geq 0$, $\theta_i \in J \Leftrightarrow i = s_k$ for some $k \geq 0$.
Set $\theta_0^{(k)} = \theta_{s_k}, x_0^{(k)} = x_{s_k}$. Recall that $T = T(\theta_0, x_0)$ was defined as the smallest positive integer satisfying that $x_T \geq \frac1{100}$.
Now, suppose that $t \geq 0$ is the smallest integer satisfying
\begin{align*}
x_{T + t} \in C, \theta_{T + t} \in \Theta_{m-1}.
\end{align*}
Since $x_T \in [1/100, 99/100]$, \cref{EverythingAligns} implies that $t \leq 2M_{m-1} + 1 < K_m \ll \sqrt{N_m}$.

Set $P = T + t \leq (N_m)^{3/4} + \sqrt{N_m} \leq 2(N_m)^{3/4} \ll N_m$, then $\theta^{(k)}_P \in \Theta_{m-1}, x^{(k)}_P \in C$ for every $k \geq 0$. Now, \cref{InCWhenReturnToBad} implies that
\begin{align*}
\psi^\K(\theta^{(k+1)}_0) = x^{(k+1)}_0 = x_{s_{k+1}} \in C,
\end{align*}
for every $k \geq 1$, or that $\psi^\K(J) \subseteq C$. Additionally, \cref{TailContraction} gives that
\begin{align*}
\prod \limits_{i=P}^{U_k - 1} |c(\theta^{(k)}_i) p'(x^{(k)}_i)| \leq (3/5)^{(U_k - P)/2}
\end{align*}
where we have set $U_j = s_{j+1} - s_j$. Since $\theta^{(k)}_P \in \Theta_{m-1}, x^{(k)}_P \in C$, \cref{SmallIteratedDerivatives} implies that
\begin{align*}
|\partial_\theta x^{(k)}_{U_k}| &= |(\partial_\theta c(\theta^{(k)}_{P-1})) p(x^{(k)}_{P-1}) + \partial_\theta x^{(k)}_P \prod \limits_{i=P}^{U_k - 1} c(\theta^{(k)}_i) p'(x^{(k)}_i) +\\
&+ \sum \limits_{j = P+1}^{U_k-1} \partial_\theta c(\theta^{(k)}_{j-1}) p(x^{(k)}_{j-1}) \prod \limits_{i = j}^{U_k-1} c(\theta^{(k)}_i) p'(x^{(k)}_i)| \leq\\
&\leq |\partial_\theta x^{(k)}_P| \cdot (3/5)^{(U_k - P)/2} + \lambda^{-1/4}.
\end{align*}
Similarly, recalling that $|c(\theta) \cdot p'(x)| \leq 4$,
\begin{align*}
|\partial_\theta x^{(k)}_P| &\leq |\partial_\theta x^{(k)}_0| \cdot \prod \limits_{i=0}^{P - 1} |c(\theta^{(k)}_i) p'(x^{(k)}_i)| +\\
&+ \|\partial_\theta c\| \left(1 + \sum \limits_{j = 1}^{P-1} \prod \limits_{i = j}^{P-1} |c(\theta^{(k)}_i) p'(x^{(k)}_i)| \right) \leq\\
&\leq |\partial_\theta x^{(k)}_0| \cdot 4^P + \|\partial_\theta c\| \sum \limits_{j = 0}^{P-1} 4^{P-1-j} =\\
&= |\partial_\theta x^{(k)}_0| \cdot 4^P + \|\partial_\theta c\| \frac{4^P - 1}{3},
\end{align*}
where $\| \cdot \|$ denotes the sup-norm. Putting it together, we obtain, since $U_k \geq N_m \gg P$, that
\begin{align*}
|\partial_\theta x^{(k)}_{U_k}| &\leq \left( |\partial_\theta x^{(k)}_0| \cdot 4^P + \|\partial_\theta c\| \frac{4^P - 1}{3} \right) (3/5)^{(U_k - P)/2} + \lambda^{-1/4} \leq\\
&\leq |\partial_\theta x^{(k)}_0| \cdot \epsilon(m) + \|\partial_\theta c\| \epsilon(m) + \lambda^{-1/4},
\end{align*}
where
\begin{align*}
\epsilon(m) = 4^P \cdot (3/5)^{(N_m - P)/2} \leq 4^P \cdot (3/5)^{N_m/2 - (N_m)^{3/4}} \to 0,
\end{align*}
as $m \to \infty$.
By induction, since $x^{(k)}_{U_k} = x_{s_{k+1}}$, we get for every $k \geq 0$ that
\begin{align*}
|\partial_\theta x_{s_{k+1}}| &\leq |\partial x^{(0)}_0| \epsilon(m)^{k+1} + \|\partial_\theta c\| \sum \limits_{j=1}^{k+1} \epsilon(m)^j + \lambda^{-1/4} \sum \limits_{j=0}^{k} \epsilon(m)^j \leq\\
&\leq \left( |\partial x^{(0)}_0| + \|\partial_\theta c\| + \lambda^{-1/4} \right) \cdot \epsilon(m) + \lambda^{-1/4}.
\end{align*}
By passing to a subsequence $\{s_{k'}\}$ of $\{s_k\}$ which satisfies $\theta_{s_{k'}} \to \theta_0$, and noting that
\begin{align*}
\partial_\theta x_{s_{k'}} &= \partial_\theta \psi^\K(\theta_{s_{k'}}) =\\
&= \partial_\theta \psi^\K(\theta_0) + \partial_\theta^2 \psi^\K(\theta_0) (\theta_{s_{k'}} - \theta_0) + o(\theta_{s_{k'}} - \theta_0) = \partial_\theta \psi^\K(\theta_0) + o(1),
\end{align*}
as $k' \to \infty$, we obtain the inequality
\begin{align*}
|\partial_\theta \psi^\K(\theta_0)| \cdot (1 - \epsilon(m)) + o(1) \leq \left(\|\partial_\theta c\| + \lambda^{-1/4} \right) \cdot \epsilon(m) + \lambda^{-1/4},
\end{align*}
which we can write as
\begin{align*}
|\partial_\theta \psi^\K(\theta_0)| \leq \lambda^{-1/4} + \epsilon'(m),
\end{align*}
for some $\epsilon'(m)$ going to 0 as $m$ goes to infinity. The proof is exactly the same for $\partial_\K \psi^\K$.

By \cref{CDerivativeAroundSecondPeak}, $\K\lambda^{1/6} < \partial_\theta c_{\alpha_c, \K = 1}(\theta) < \K\lambda$ for every $\theta \in I_0+ \omega$.
When $\theta \in J$, then $\psi^\K(\theta) \in C$. Therefore
\begin{align*}
\frac{3}{10} < \frac32 \cdot p(1/3 + 1/100) \leq p(\psi^\K(\theta)) \leq 4p(1/3 + 1/100) < 95/100.
\end{align*}
Recall that $|\partial_\theta \psi^\K(\theta)| < (1 + \epsilon(m)) \lambda^{-1/4}$, where $\epsilon(m) \to 0$ as $m \to \infty$. Since
\begin{align*}
\partial_\theta \psi^\K(\theta + \omega) = \left(\partial_\theta c(\theta) \right) \cdot p(\psi^\K(\theta)) + c(\theta) \cdot p'(\psi^\K(\theta)) \cdot \partial_\theta \psi^\K(\theta),
\end{align*}
assuming that $\lambda$ is very large, we obtain after a straight-forward computation that
\begin{align*}
\K \lambda^{1/7} < \partial_\theta \psi^\K(\theta + \omega) < \K \lambda.
\end{align*}
\end{proof}

\begin{corr}\label{BigDerivative}
There is an $n_0 \geq 0$ such that, for every $n \geq n_0$, and every $\K \in B_n \backslash B_{n-1}$ (sufficiently close to 1)
\begin{align*}
\K \lambda^{1/7} \leq \partial_\theta \psi^\K(\theta) \leq \K \lambda
\end{align*}
for every $\theta \in I_n + \omega$, assuming that $\lambda > 0$ is sufficiently large. Moreover
\begin{align*}
\psi^\K(I_n) \subseteq C.
\end{align*}
\end{corr}

\begin{proof}
Let $0 \leq \K < 1$ sufficiently close to 1 be given, and choose $J = I_{n}$, where $n = n(\K)$. Now \cref{QuickReturnFromWorstToGood} tells us that
\begin{align*}
x_{2K_n + 20} \in C, \theta_{2K_n + 20} \in \Theta_{n-1},
\end{align*}
that is $\max \limits_{\theta \in I_n} T(\theta_0) \leq 2K_n + 20$, where $K_n \ll \sqrt{N_n} \leq (N_n)^{3/4}$. Both statements now follow immediately from \cref{DerivativeBoundBeforePeak}.
\end{proof}

\begin{lemma}\label{ParamDerivativeAtAlpha}
There is an $0 < \epsilon \leq 1$ such that, for every $1 - \epsilon \leq \K < 1$,
\begin{align*}
\frac1{3} \leq \partial_\K \psi^\K(\alpha_c) \leq \frac52,
\end{align*}
provided that $\lambda > 0$ is sufficiently large. Moreover,
\begin{align*}
\lim \limits_{\K \to 1^-} \psi^\K(\alpha_c) = 1/2,
\end{align*}
and, as $\K \to 1^-$,
\begin{align}
|\psi^\K(\alpha_c) - 1/2| = O(1-\K).\label{DistanceFromOneHalfForAlphaC}
\end{align}
\end{lemma}

\begin{proof}
For $\K$ sufficiently close to 1, \cref{BigDerivative} implies that $\psi^\K(I_n) \subseteq C$, and that $|\partial_\K \psi^\K(\theta)| < \lambda^{1/4} + \epsilon(n)$ for $\theta \in I_n$,
where $\epsilon(n) \to 0$ as $n \to \infty$.
By invariance of $\psi^\K$ under the map $\Phi_{\alpha_c, \K}$,
\begin{align*}
\partial_\K \psi(\alpha_c) = \partial_\K c(\alpha_c - \omega) p(\psi(\alpha_c - \omega) + c(\alpha_c - \omega) \partial_\K \psi^\K(\alpha_c - \omega).
\end{align*}
By definition of the set $\mathcal{A}_0 \ni \alpha_c$, $2\lambda^{-2/3} \leq 0 - (\alpha_c - \omega) \leq \lambda^{-2/5}/2$, which means that
\begin{align*}
c(\alpha_c - \omega) - c(0) = \lambda^{-2/5}/2 \partial_\theta c(0) + o(\lambda^{-2/5}) = o(\lambda^{-2/5}),
\end{align*}
or that $c(\alpha_c - \omega) = \frac32 + \K \frac52 + o(\lambda^{-2/5})$. This implies that $\partial_\K c(\alpha_c - \omega) = \frac52 + o(\lambda^{-2/5})$.
Therefore
\begin{align*}
(\frac52 + o(\lambda^{-2/5})) (\frac13 - \frac1{100}) + \lambda^{-1/4} + \epsilon(n) \leq \partial_\K \psi(\alpha_c) \leq (\frac52 + o(\lambda^{-2/5})) (\frac13 + \frac1{100}) + 4\lambda^{-1/4} + 4\epsilon(n),
\end{align*}
or
\begin{align*}
\frac13 \leq 2\frac13 + o(\lambda^{-1/10}) + \epsilon(n) \leq \partial_\K \psi(\alpha_c) \leq \frac54 + o(\lambda^{-1/10}) + 4\epsilon(n) \leq \frac52,
\end{align*}
if $n$ and $\lambda$ are suffciently large.

Suppose that $\theta_0 = \alpha_c - M_n \omega$, $x_0 \in C$. In \cite{BjerkSNA}, it was proved that, if $\K = 1$, then
\begin{align*}
\lim \limits_{n \to \infty} x_{M_n} = 1/2.
\end{align*}
Letting $x_{M_n}(\K)$ (a smooth function in $\K$) be as above, but corresponding to a $\K \in [0,1]$ sufficiently close to 1, we obtain uniform bounds on
\begin{align*}
\partial_\K x_{M_n}(\K).
\end{align*}
Since
\begin{align*}
x_{M_n}(\K) - x_{M_n}(1) = \partial_\K x_{M_n}(\tilde{\K})(\K - 1),
\end{align*}
for some $\K \leq \tilde{\K} \leq 1$, we have, for large enough $n \geq 0$,
\begin{align*}
|x_{M_n}(\K) - 1/2| &= |(x_{M_n}(\K) - x_{M_n}(1)) + (x_{M_n}(1) - 1/2)| < 2\epsilon,
\end{align*}
uniformly in $n$, for $\K$ sufficiently close to 1. From this, it follows that
\begin{align*}
\lim \limits_{\K \to 1^-} \psi^\K(\alpha_c) = \lim \limits_{n \to \infty} \lim \limits_{\K \to 1^-} x_{M_n}(\K) = 1/2.
\end{align*}
By the mean value theorem
\begin{align*}
\psi^\K(\alpha_c) = \lim \limits_{\tilde{\K} \to 1^-} \psi^{\tilde{\K}}(\alpha_c) + \partial_\K \psi^{\tilde{\K}}(\alpha_c)(\K - \tilde{\K}) + o(\K - \tilde{\K}) = 1/2 + O(1 - \K),
\end{align*}
since $\frac1{3} \leq \partial_\K \psi^\K(\alpha_c) \leq \frac52$.
\end{proof}

\begin{dfn}
Let $T_1(\K, \theta)$ be defined, for every $\theta \in I_0 + 3\omega$,
as the smallest integer $0 \leq T_1(\K, \theta)$ such that $\psi^\K(\theta + T_1(\K, \theta) \cdot \omega) \geq \frac1{100}$.
\end{dfn}

By its very definition $\max \limits_{\theta \in \mathbb{T}} T_1(\K, \theta) \leq M_C(\K)$, where $M_C(\K)$ is the constant appearing in \cref{Bn}. Hence, if
\begin{align*}
2K_{n-1} - 2 < \max \limits_{\theta \in I_0 + 3\omega} T_1(\K, \theta) \leq 2K_n - 2,
\end{align*}
then $\K \in B_n \backslash B_{n-1}$. Set
\begin{align*}
T_1(\K) = \max \limits_{\theta \in \mathbb{T}} T_1(\K, \theta).
\end{align*}

\begin{prop}\label{MinimumDistance}
Suppose that $\K < 1$ is sufficiently close to 1, and that $\K \in B_n \backslash B_{n-1}$, i.e. that
\begin{align*}
2K_{n-1} - 2 < T_1(\K) \leq 2K_n - 2.
\end{align*}
Then the the minimum distance between the repelling set and the attractor is attained $I_n + 3\omega$, and is asymptotically linear in $\K$. Specifically,
\begin{align}
\delta(\K) = c_{\K = 1}(\alpha_c + \omega) \cdot \frac58(1 - \K) + o(1-\K)\label{MinimumValue}
\end{align}
asymptotically as $\K \to 1^-$. Moreover,
\begin{align}
\psi^\K(\alpha_c) = \frac38 + \K\frac58 + o(1-\K).\label{PsiInAlphaC}
\end{align}
\end{prop}

\begin{proof}
If $\psi(\theta) \in (a, 1/10)$, where $0 \leq a < 1/10$, then $4 \psi(\theta) \geq \psi(\theta + \omega) \geq \frac54 \psi(\theta)$ (see \cref{AscentFromBottom}),
or $\psi(\theta + \omega) \in [\frac54 a, 99/100]$.
Similarly, if $\psi(\theta) \in (9/10, b)$, where $9/10 < b \leq 1$, then $\psi(\theta + \omega) \in (\frac54(1-b), 99/100)$ (since $p(1 - x) = p(x)$).

As long as $\theta \not\in I_0 \cup (I_0 + \omega)$, then $\psi(\theta) \in [1/100, 99/100]$ implies that $\psi(\theta + \omega) \in [1/100, 2/5] \subset [1/100, 99/100]$
(see \cref{CloseToCIfAwayFromPeak}).

One implication of this is, that a value strictly greater than $99/100$ can never be attained for a $\theta \not\in (I_0 + \omega) \cup (I_0 + 2\omega)$.
Another one is that, if a value strictly less than 1/100 is attained, the minimum has to be attained in the iteration immediately following a value greater than 99/100,
i.e., for $\theta \in (I_0 + 2\omega) \cup (I_0 + 3\omega)$.

This means that we only need to analyze $\psi^\K(\theta)$ for $\theta \in (I_0 + \omega) \cup (I_0 + 2\omega) \cup (I_0 + 3\omega)$.

We know that the part of $\psi^\K$ lying below $1/100$ even in these intervals will rise with each iteration, meaning that the lowest part,
the one closes to 0, must come from a previous value strictly greater than 99/100. Therefore, we are interested in seeing how far above 99/100
$\psi^\K$ can get.

By the above discussion, necessarily $\psi(\theta) \leq 2/5$ for $\theta \in I_0$, and so the theoretical maximum for $I_0 + \omega$ is
\begin{align*}
\psi^\K(\theta) \leq 4 p(2/5) = 24/25.
\end{align*} 
The theoretical minimum coming from that is at least $\geq 1/25$. Thus, we turn to $I_0 + 2\omega$.

By \cref{DistanceFromOneHalfForAlphaC},
\begin{align*}
|\psi^\K(\alpha_c) - 1/2| = O(1-\K).
\end{align*}
Therefore
\begin{align*}
\psi^\K(\alpha_c + \omega) = c(\alpha_c) p(1/2 + O(1-\K)) = \left(\frac32 + \K \frac52 \right) \left(\frac14 + O((1-\K)^2) \right) = \frac38 + \K \frac58 + o(1-\K),
\end{align*}
and
\begin{align*}
1 - \psi^\K(\alpha_c + \omega) = \frac58 (1-\K) + o(1-\K).
\end{align*}
Note that this maximum is, up to the error term $o(1-\K)$, equal to the theoretical maximum $c(\alpha_c)p(1/2)$.
Therefore, the minimum is at most
\begin{align}
\psi^\K(\alpha_c + 2\omega) = c_\K(\alpha_c + \omega) p(\psi^\K(\alpha_c + \omega)) \leq 4 (1 - \psi^\K(\alpha_c + \omega)) \leq \frac52 (1-\K) + o(1-\K),\label{ValueAtAlphaCPlus2Omega}
\end{align}
and at least ($\theta \in I_n + 2\omega$)
\begin{align*}
\psi^\K(\theta + \omega) \geq \frac54 (1 - \psi^\K(\theta)) \geq 1 - \psi^\K(\theta) \geq \frac58 (1-\K) \geq \frac12(1-\K),
\end{align*}
for $\K$ sufficiently close to 1. More specifically, we have
\begin{align*}
\psi^\K(\theta + \omega) = c_\K(\theta)\psi^\K(\theta)(1 - \psi^\K(\theta)).
\end{align*}
There is some $\tilde{\theta}$ between $\theta$ and $\alpha_c$, such that
\begin{align*}
\psi^\K(\theta + \omega) = \psi^\K(\alpha_c) + \partial_\theta \psi^\K(\tilde{\theta})(\theta - \alpha_c).
\end{align*}
A quick Taylor expansion gives
\begin{align*}
p(y) = p(x) + (1 - 2x)(y - x) - (y - x)^2.
\end{align*}
Since $c(\theta) = c(\alpha_c) + \partial^2_\theta c(\alpha_c) (\theta - \alpha_c)^2 + o( (\theta - \alpha_c)^2)$ for $\theta$ very close to $\alpha_c$,
such as for $\theta \in I_n + \omega$, and $\psi^\K(\alpha_c) = 1/2 - \partial_\K \psi^{\tilde{\K}}(\alpha_c)(1-\K)$ for some $\tilde{\K}$ between 1 and $\K$,
this means that
\begin{align*}
\psi^\K(\theta + \omega) &= \left( c_\K(\alpha_c) + o(\theta - \alpha_c) \right) \left( \psi^\K(\alpha_c) + \left( -2\partial_\K \psi^{\tilde{\K}}(\alpha_c)(1-\K) \right) \partial_\theta \psi^\K(\tilde{\theta})(\theta - \alpha_c) +
o(\theta - \alpha_c) \right) =\\
&= \psi^\K(\alpha_c + \omega) - A_2(\K, \theta) \cdot (\theta - \alpha_c),
\end{align*}
for some constant $A_2(\K, \theta) > \frac12 \lambda^{-1/7}$, since $\partial_\theta \psi^\K(\tilde{\theta}) \geq \K\lambda^{1/7}$ (see \cref{BigDerivative})
and $\partial_\K \psi^{\tilde{\K}}(\alpha_c) \geq \frac13$ (see \cref{ParamDerivativeAtAlpha}).
Similarly, in the next iteration, we obtain
\begin{align*}
\psi^\K(\theta + 2\omega) &= c_\K(\theta + \omega) \left( p(\psi^\K(\alpha_c + \omega)) - (1 + O(1-K))(-A_2(\K, \theta) (\theta - \alpha_c) + o(\theta - \alpha_c) \right) =\\
&= c_\K(\theta + \omega) \left( p(\psi^\K(\alpha_c + \omega)) + A_2(\K, \theta) (\theta - \alpha_c) + o(\theta - \alpha_c) + o(1-\K) \right).
\end{align*}
Since $c_\K(\theta + \omega) = c_\K(\alpha_c + \omega) + A_3(\theta)(\theta - \alpha_c) + o(\theta - \alpha_c)$,
where $|A_3(\theta)| = |\partial_\theta c_\K(\alpha_c + \omega)| \leq \lambda^{-1/2}$ (see \cref{CDerivativeOutsideI0}), this reduces to
\begin{align*}
\psi^\K(\theta + 2\omega) &= c_\K(\alpha_c + \omega) \left( p(\psi^\K(\alpha_c + \omega)) + A_2(\K, \theta) (\theta - \alpha_c) + o(\theta - \alpha_c) + o(1-\K) \right) +\\
&+ A_3(\theta)(\theta - \alpha_c) \left( p(\psi^\K(\alpha_c + \omega)) + A_2(\K, \theta) (\theta - \alpha_c) + o(\theta - \alpha_c) + o(1-\K) \right) =\\
&= c_\K(\alpha_c + \omega) p(\psi^\K(\alpha_c + \omega)) + K_4(\K, \theta)(\theta - \alpha_c) + o(1-\K),
\end{align*}
where $K_4(\K,\theta) > 0$. This gives us immediately the asymptotic on the distance, since $p(\psi^\K(\alpha_c + \omega)) = \frac58(1-\K) + o(1-\K)$, as shown above.

If we can prove that no point outside $I_n + 2\omega$ reaches as high as this, we are done. Recall \cref{TimeOfAscent}, stating that
\begin{align*}
T_1(\K) \leq \max \limits_{\theta \in I_n} \log_{5/4} \frac1{20\psi(\theta + 3\omega)} \leq \log_{5/4} \frac1{10(1-\K)}.
\end{align*}
This of course means that
\begin{align*}
2K_{n-1} - 2 \leq T_1(\K) \leq \log_{5/4} \frac1{10(1-\K)}.
\end{align*}
By definition of $I_n$, $|I_n| = (4/5)^{K_{n-1}}$, or
\begin{align*}
|I_n| = (4/5)^{K_{n-1}} \geq (4/5)^{T_1(\K)/2 + 1} \geq \frac45 \sqrt{10(1-\K)} \geq 2 \sqrt{1-\K}.
\end{align*}
Since $I_n$ is centred at $\alpha_c$, this means that
\begin{align*}
[\alpha_c - \sqrt{\lambda^{1/2(1-\K)}}\lambda^{-1/4}, \alpha_c + \sqrt{\lambda^{1/2(1-\K)}}\lambda^{-1/4}] \subset I_n.
\end{align*}
Invoking \cref{AlphaPeakZoom}, we obtain the set inclusion
\begin{align*}
\{ \theta \in I_0 + \omega : c_{\alpha_c, \K}(\theta) \geq \left( \frac32 + \K \frac52 \right)(1 - \lambda^{1/2}(1 - \K)) \} \subset I_n.
\end{align*}
Hence, the theoretical maximum attained for $\theta \in (I_0 \backslash I_n) + 2\omega$ is
\begin{align*}
\psi^\K(\theta) < \left( \frac32 + \K \frac52 \right)(1 - \lambda^{1/2}(1 - \K)) p(1/2) = \left( \frac38 + \K \frac58 \right)(1 - \lambda^{1/2}(1 - \K)),
\end{align*}
which is by the order of $1-\K$ less than the maximum in $I_n$. Hence, the minimum for $\theta + \omega \in (I_0 \backslash I_n) + 3\omega$ satisfies
\begin{align*}
\psi^\K(\theta + \omega) \geq \frac54 (1 - \psi^\K(\theta)) \geq \left( \frac58 + \frac38\lambda^{1/2} \right)(1 - \K) > \psi^\K(\alpha_c + 2\omega),
\end{align*}
which is bigger than the minimum attained in $I_n + 3\omega$.
\end{proof}

\subsection{Asymptotic growth of the maximum derivative of the attractor}

The basic idea in this section is that the derivative in the interval $I_{n(\K)} + \omega$, which is centered at $\alpha_c$ where 1/2 is almost attained, is large and approximately linear in $\K$.
In the next iteration, this means that this segment becomes approximately quadratic around the maximum point, which is almost at$\alpha_c + \omega$.
The approximately quadratic shape around the minimum point (almost $\alpha_c + 2\omega$) is retained in the next ieration.

The derivative at a point $\theta + 2\omega \in I_{n(\K)} + 3\omega$ will be approximately equal to $(\theta - \alpha_c)$,
and the value $\psi^\K(\theta + 2\omega)$ will be approximately $(1-\K) + (\theta - \alpha_c)^2$.

Expanding the derivative at $\theta + (2 + T)\omega$ as a recurrence relation (as we have done several times before), the dominant term as $T$ grows will behave like
\begin{align*}
\partial_\theta \psi^\K(\theta + 2\omega) \cdot \prod_{k=0}^{T} c(\theta_k) \cdot p'(x_k) \sim \frac{\partial_\theta \psi^\K(\theta + 2\omega)}{\psi^\K(\theta + 2\omega)}
\sim \frac{\theta - \alpha_c}{(1-\K) + (\theta - \alpha_c)^2},
\end{align*}
when $T = T_1(\K, \theta)$ (see \cref{DefinitionOfT1} for the definition).

In practice, we will work with a slightly enlarged set $J_\K \supseteq I_{n(\K)} + \omega$ which is centered at $\alpha_c$. This set will be of size $\gtrsim \sqrt{1-\K}$.
This allows us to choose $(\theta - \alpha_c) \sim \sqrt{1-\K}$, which maximizes
\begin{align*}
\frac{\theta - \alpha_c}{(1-\K) + (\theta - \alpha_c)^2} \sim \frac1{\sqrt{1-\K}}.
\end{align*}
The last step is showing that the derivative can't grow much more.
The worst case would be when we get close to the peak only a few iterations after $T_1(\K, \theta)$ (when we have come back to the contracting region),
potentially causing the derivative to grow further.

If this were to occur, we would only visit parts so far from the peaks that it wouldn't have much effect on the derivative,
since we would need a much longer time to get back to the "worst parts" of the peaks. We show this by considering two cases:
\begin{itemize}
\item We just recently changed scales from some $I_m$ to $I_{m+1}$ (due to an increase in $\K$).
In this case, we show that actually we may work with $I_m$, as if it were the appropriate scale,
having all the constants work to our advantage (which they wouldn't have, had we been forced to work with $I_{m+1}$).
\item We changed scales a long time ago, meaning that $\frac1{\sqrt{1-\K}}$ is large enough to withstand the relatively small
products coming from having come close to the peak, even the ones using the estimates that were inappropriate in the former case.
\end{itemize}
This last bit is the contents of \cref{DerivativeGrowth}, the main result in this section.

\begin{lemma}\label{DerivativeGrowsDuringExpansion}
There is a constant $K > 0$ such that if $|\partial_\theta x_0| \geq K$ and $x_0 \leq \frac1{100}$, then for any $0 \leq \K < 1$
\begin{align*}
|\partial_\theta x_1| > |\partial_\theta x_0|.
\end{align*}
\end{lemma}

\begin{proof}
Since $x_0 \leq 1/100$, $c_\K(\theta_0) p'(x_0) \geq \frac32 (1 - \frac2{100}) \geq \frac54$. Now
\begin{align*}
|\partial_\theta x_1| &= | (\partial_\theta c_\K(\theta_0)) \cdot p(x_0) + c_\K(\theta_0) \cdot p'(x_0) \cdot \partial_\theta x_0| \geq\\
&\geq |c_\K(\theta_0) \cdot p'(x_0) \cdot \partial_\theta x_0| - |\partial_\theta c_\K(\theta) \cdot p(x_0)| \geq\\
&\geq \frac54 \cdot |\partial_\theta x_0| - |\partial_\theta c_\K(\theta)|.
\end{align*}
If $|\partial_\theta x_0|$ is sufficiently large, the conclusion follows.
\end{proof}

Recall that we defined $T_1(\K, \theta)$, for $\theta \in I_{n(\K)} + 3\omega$, as the smallest integer $0 \leq T_1(\K, \theta)$ such that $\psi^\K(\theta + T_1(\K, \theta) \cdot \omega) \geq \frac1{100}$.
Set
\begin{align}
T_1(\K) = \max \limits_{\theta \in I_{n(\K)} + 3\omega} T_1(\K, \theta).\label{DefinitionOfT1}
\end{align}

\begin{lemma}\label{MayChooseSquareRoot}
When $\K < 1$ is sufficiently close to 1, the following holds:

If $2K_{n-1} - 2 < T_1(\K) \leq 2K_n - 2$, then there is an interval $J_\K \subseteq I_n + 2\omega$, centered at the point $\alpha_c$, satisfying
\begin{align}
\K\lambda^{1/7} \leq \partial_\theta \psi^\K(\theta) \leq \K \lambda,\label{JDerivativeInequality}
\end{align}
for every $\theta \in J_\K$, and
\begin{align*}
|J_\K| \geq \frac45 (\sqrt{1 - \K})^{1/\eta},
\end{align*}
where $\eta = \frac{T_1(\K)}{2K_{n-1} - 2} > 1$.
\end{lemma}

\begin{proof}
By \cref{TimeOfAscent},
\begin{align*}
T_1(\K) = \max \limits_{\theta \in I_{n(\K)} + 3\omega} \log_{5/4} \frac1{20\psi^\K(\theta)}.
\end{align*}
Now, \cref{MinimumValue} implies that
\begin{align*}
\min \limits_{\theta \in I_{n(\K)} + 3\omega} \psi^\K(\theta) \geq \frac32 \cdot \frac58(1-\K) + o(1-\K) \geq \frac12(1-\K).
\end{align*}
Therefore
\begin{align*}
T_1(\K) \leq \log_{5/4} \frac1{10(1-\K)}.
\end{align*}
\Cref{BigDerivative} implies that any such $J_\K$ can include at least the interval $I_n$, which is centered at $\alpha_c$.
Now, recalling that $T_1(\K) = \eta (2K_{n-1} - 2)$, or $K_{n-1} = \frac{T_1}{2\eta} + 1$, we get
\begin{align*}
|I_n| = (4/5)^{K_{n-1}} > (4/5)^{T_1(\K)/(2\eta) + 1} \geq (4/5) \cdot (\sqrt{10 (1-\K)})^{1/\eta} \geq (4/5) \cdot (\sqrt{1-\K})^{1/\eta}.
\end{align*}
Hence $J_\K = I_n$ satisfies the conclusions.
\end{proof}

From this point on, let $J_\K$ denote the largest interval centered at $\alpha_c$, and satisfying the conclusion in \cref{MayChooseSquareRoot}.

\begin{lemma}
Suppose that $0 \leq \K < 1$, and $n = n(\K)$. If $T_1(\K) \leq K_{n-1}^{3/2}(2K_{n-1} - 2)$, then 
\begin{align*}
J_\K \supseteq I_{n-1} + \omega.
\end{align*}
\end{lemma}

\begin{proof}
By our assumptions on $T_1(\K)$,
\begin{align*}
T_1(\K) \leq K_{n-1}^{3/2}(2K_{n-1} - 2) \ll K_{n-1}^3 \sim (N_{n-1})^{3/4}.
\end{align*}
Applying \cref{DerivativeBoundBeforePeak} to the set $J = I_{n-1}$, the statement follows, since every $T_1(\K, \theta) \leq T_1(\K) \leq (N_{n-1})^{3/4}$.

Since $I_{n-1}$ is centered at $\alpha_c$, the set $I_{n-1}$ satisfies the conclusions in \cref{MayChooseSquareRoot}.
\end{proof}

We recall that $J_\K + \omega$ is where the maximum of the graph is located, and $J_\K + 2\omega$ will be the location of the minimum.

\begin{lemma}\label{DerivativeAfterPeak}
Assume that $0 \leq \K < 1$ is sufficiently close to 1. Then there are numbers $\frac1K < A_1(\K, \theta) < K, \frac1K < A_2(\K,\theta) < K$ (where $K > 0$),
depending only on $\K$ and $\theta$, such that, for every $\theta \in J_\K$,
\begin{itemize}
\item $\partial_\theta \psi^\K(\theta + \omega) = -A_1(\K,\theta) \cdot (\theta - \alpha_c) + O(1-\K)$, and
\item $\partial_\theta \psi^\K(\theta + 2\omega) = A_2(\K,\theta) \cdot (\theta - \alpha_c) + O(1-\K)$.
\end{itemize}
\end{lemma}

\begin{proof}
Throughout this entire proof, we will make use of the previous result that $\psi^\K(\alpha_c) = 1/2 + O(1-\K)$ (see \cref{DistanceFromOneHalfForAlphaC}).

Let $\theta +\omega \in J_\K + \omega$ be arbitrary. We have the usual recurrence relation
\begin{align*}
\partial_\theta \psi^\K(\theta + \omega) &= \partial_\theta c_\K(\theta) \cdot p(\psi^\K(\theta)) + c_\K(\theta) \cdot p'(\psi^\K(\theta)) \cdot \partial_\theta \psi^\K(\theta).
\end{align*}
We will analyze each term in detail, starting with
\begin{align*}
\partial_\theta c_\K(\theta) &= \partial_\theta c_\K(\alpha_c) + \partial_\theta^2 c_\K(\alpha_c) (\theta - \alpha_c) + o(\theta - \alpha_c) =\\
&= \partial_\theta^2 c_\K(\alpha_c) (\theta - \alpha_c) + o(\theta - \alpha_c),
\end{align*}
where $\partial_\theta^2 c_\K(\alpha_c) \leq 0$, since $\alpha_c$ is a local maximum for $c_\K$.
\begin{align*}
p(\psi^\K(\theta)) &= p(\psi^\K(\alpha_c)) + p'(\psi^\K(\alpha_c)) \partial_\theta \psi^\K(\alpha_c) (\theta - \alpha_c) + o(\theta - \alpha_c) =\\
&= 1/4 + o(1-\K) + O(1-\K)O(\theta - \alpha_c) + o(\theta - \alpha_c) = 1/4 + o(1-\K) + o(\theta - \alpha_c),
\end{align*}
since $\theta - \alpha_c = o(1)$ as $\K \to 1^-$ (they lie in successively smaller intervals $I_{n(\K)}$).
Putting it together, the effects of the first term is:
\begin{align*}
\partial_\theta c_\K(\theta) p(\psi^\K(\theta)) = (1/4) \partial_\theta^2 c_\K(\alpha_c) (\theta - \alpha_c) + o(1-\K) + o(\theta - \alpha_c).
\end{align*}
The second term can be similarly analyzed, starting with
\begin{align*}
p'(\psi^\K(\theta)) &= p'(\psi^\K(\alpha_c) + \partial_\theta \psi^\K(\alpha_c)(\theta - \alpha_c) + o(\theta - \alpha_c)) =\\
&= 1 - 2(1/2 + O(1-\K) + \partial_\theta \psi^\K(\alpha_c)(\theta - \alpha_c) + o(\theta - \alpha_c) =\\
&= O(1-\K) - 2\partial_\theta \psi^\K(\alpha_c)(\theta - \alpha_c) + o(\theta - \alpha_c).
\end{align*}
Therefore
\begin{align*}
c_\K(\theta) \cdot p'(\psi^\K(\theta)) \cdot \partial_\theta \psi^\K(\theta) = - 2\partial_\theta \psi^\K(\alpha_c)\partial_\theta \psi^\K(\theta)(\theta - \alpha_c) + O(1-\K) + o(\theta - \alpha_c).
\end{align*}
We thus obtain the equality
\begin{align*}
\partial_\theta \psi^\K(\theta + \omega) = (1/4) \partial_\theta^2 c_\K(\alpha_c) (\theta - \alpha_c) - 2\partial_\theta \psi^\K(\alpha_c)\partial_\theta \psi^\K(\theta)(\theta - \alpha_c) + O(1-\K) + o(\theta - \alpha_c),
\end{align*}
or, recalling that $\partial_\theta^2 c_\K(\alpha_c) \leq 0$ and $\partial_\theta \psi^\K(\alpha_c)\partial_\theta \psi^\K(\theta) > \K^2(\lambda^{1/7})^2$, the bounds
\begin{align*}
\partial_\theta \psi^\K(\theta + \omega) = -A_1(\K, \theta) (\theta - \alpha_c) + O(1-\K),
\end{align*}
where $\frac1K < A_1(\K, \theta) < K$, for some $K > 0$, as $\K \to 1^-$.

In the next iteration, for $\theta + 2\omega \in J_\K + 2\omega$, we have
\begin{align*}
\partial_\theta \psi^\K(\theta + 2\omega) &= \partial_\theta c_\K(\theta + \omega) \cdot p(\psi^\K(\theta + \omega)) + c_\K(\theta + \omega) \cdot p'(\psi^\K(\theta + \omega)) \cdot \partial_\theta \psi^\K(\theta + \omega).
\end{align*}
The first term is $O(1-\K) + o(\alpha_c - \theta)$, since
\begin{align*}
p(\psi^\K(\theta + \omega)) &= p(\psi^\K(\alpha_c + \omega)) + p'(\psi^\K(\alpha_c + \omega)) \partial_\theta \psi^\K(\alpha_c + \omega) (\theta - \alpha_c) + o(\theta - \alpha_c) =\\
&= O(1-\K) + (O(\theta - \alpha_c) + O(1-\K))(\theta - \alpha_c) + o(\theta - \alpha_c).
\end{align*}
For the second term, note that
\begin{align*}
p'(\psi^\K(\theta + \omega)) &= p'(\psi^\K(\alpha_c + \omega) + \partial_\theta \psi^\K(\alpha_c + \omega)(\theta - \alpha_c) + o(\theta - \alpha_c)) =\\
&= 1 - 2(\psi^\K(\alpha_c + \omega) + \partial_\theta \psi^\K(\alpha_c + \omega)(\theta - \alpha_c) + o(\theta - \alpha_c) =\\
&= -\psi^\K(\alpha_c + \omega) + O(1-\K) + O(\theta - \alpha_c) =\\
&= -\left( \frac38 + \K\frac58 \right) + O(1-\K) + O(\theta - \alpha_c),
\end{align*}
resulting in (note the cancellation of signs!), by the previous estimate of $\partial_\theta \psi^\K(\theta + \omega)$,
\begin{align*}
\partial_\theta \psi^\K(\theta + 2\omega) &= c(\theta + \omega) (-)\left( \frac38 + \K\frac58 \right) (-)A_1(\K, \theta) (\theta - \alpha_c) + O(1-\K) + o(\theta - \alpha_c) =\\
&= c(\theta + \omega) \left( \frac38 + \K\frac58 \right) A_1(\K, \theta) (\theta - \alpha_c) + O(1-\K) + o(\theta - \alpha_c) =\\
&= A_2(\K, \theta) (\theta - \alpha_c) + O(1-\K),
\end{align*}
where $\frac1K < A_2(\K, \theta) < K$ (for some $K > 0$).
\end{proof}

The lemma below says that the attracting curve is approximately quadratic around $\theta_{\max} + \omega$ (approximately where the global minimum is located).
If we could control the higher derivatives sufficiently well, the proof would have been very straightforward.

\begin{lemma}\label{DifferenceInValuesAfterPeak}
Suppose that $0 \leq \K < 1$ is sufficiently close to 1. Then there is a number $\frac1K < A_3(\K, \theta) < K$ (where $K > 0$),
depending only on $\K$ and $\theta$, such that
\begin{align*}
\psi^\K(\theta + 2\omega) - \psi^\K(\alpha_c + 2\omega) = -A_3(\K, \theta)(\alpha_c - \theta)^2 + o(1-\K),
\end{align*}
for every $\theta + 2\omega \in J_\K + 2\omega$.
\end{lemma}

\begin{proof}
We remind ourselves that $\psi^\K(\alpha_c + \omega) = \frac38 + \K\frac58 + o(1-\K)$ (see \cref{PsiInAlphaC}),
and therefore $1 - \psi^\K(\alpha_c + \omega) = \frac58(1 - \K) + o(1-\K) = O(1-\K)$.
We also remind ourselves that $\psi^\K(\alpha_c) = 1/2 + O(1-\K)$ (see \cref{DistanceFromOneHalfForAlphaC}).

We begin by analyzing the differences
\begin{align*}
\psi^\K(\alpha_c + \omega) - \psi^\K(\theta + \omega),
\end{align*}
where $\theta + \omega \in J_\K + \omega$. Now
\begin{align*}
\psi^\K(\alpha_c + \omega) - \psi^\K(\theta + \omega) &= c_\K(\alpha_c) p(\psi^\K(\alpha_c)) - c_\K(\theta) p(\psi^\K(\theta)) =\\
&= c_\K(\theta) \left( p(\psi^\K(\alpha_c)) - p(\psi^\K(\theta)) \right) +\\
&+ \left( c_\K(\alpha_c) - c_\K(\theta) \right) p(\psi^\K(\alpha_c)).
\end{align*}
We know that
\begin{align*}
\psi^\K(\alpha_c) - \psi^\K(\theta) &= \partial_\theta \psi^\K(\alpha_c) (\alpha_c - \theta) + o(\theta - \alpha_c).
\end{align*}
A quick Taylor expansion gives
\begin{align*}
p(y) - p(x) = (1 - 2x)(y - x) - (y - x)^2.
\end{align*}
Now,
\begin{align*}
p(\psi^\K(\alpha_c)) - p(\psi^\K(\theta)) &= (1 - 2\psi^\K(\alpha_c)) \partial_\theta \psi^\K(\alpha_c) (\alpha_c - \theta) +
\left( \partial_\theta \psi^\K(\alpha_c) \right)^2 (\alpha_c - \theta)^2 + o((\alpha_c - \theta)^2) =\\
&= o(1-\K) + \left( \partial_\theta \psi^\K(\alpha_c) \right)^2 (\alpha_c - \theta)^2 + o((\alpha_c - \theta)^2),
\end{align*}
since $(1 - 2\psi^\K(\alpha_c)) = 1 - 2 \cdot (1/2 + O(1-\K)) = O(1-\K)$, and $(\theta - \alpha_c) = o(1)$ as $\K \to 1^-$ (the interval $I_{n(\K)}$ shrinks).
Hence, the first term is
\begin{align*}
c_\K(\theta) \left( p(\psi^\K(\alpha_c)) - p(\psi^\K(\theta)) \right) = c_\K(\theta) \left( \partial_\theta \psi^\K(\alpha_c) \right)^2 (\alpha_c - \theta)^2 + o((\alpha_c - \theta)^2) + o(1-\K).
\end{align*}
Taylor series expansions around $\alpha_c$ yield, since $\partial_\theta c(\alpha) = 0$,
\begin{align*}
c_\K(\alpha_c) - c_\K(\theta) &= - \left( \partial_\theta c_\K(\alpha_c) (\theta - \alpha_c) + \partial^2_\theta c_\K(\alpha_c) (\theta - \alpha_c)^2 + o((\theta - \alpha_c)^2) \right) =\\
&= -\partial^2_\theta c_\K(\alpha_c) (\theta - \alpha_c)^2 + o((\theta - \alpha_c)^2),
\end{align*}
where for some constant $K \leq 0$, $0 \leq -\partial^2_\theta c_\K(\alpha) \leq K$ for all $0 \leq \K < 1$, since $c_\K(\theta)$ has a local maximum at $\alpha$.
Therefore, the total effect is
\begin{align*}
\psi^\K(\alpha_c + \omega) - \psi^\K(\theta + \omega) = K(\theta, \K)(\alpha_c - \theta)^2 + o(1-\K)
\end{align*}
where $K(\theta, \K) = \left( \partial_\theta \psi^\K(\alpha_c) \right)^2 - \partial^2_\theta c_\K(\alpha_c)$ satisfies $\frac1K < K(\theta, \K) < K$ (see \cref{BigDerivative}) for some $K > 0$.
Turning to the next iteration (the one we are interested in), where $\theta + 2\omega \in J_\K + 2\omega$, we have
\begin{align*}
\psi^\K(\alpha_c + 2\omega) - \psi^\K(\theta + 2\omega) &= c_\K(\theta + \omega) \left( p(\psi^\K(\alpha_c + \omega)) - p(\psi^\K(\theta + \omega)) \right) +\\
&+ \left( c_\K(\alpha_c + \omega) - c_\K(\theta + \omega) \right) p(\psi^\K(\alpha_c + \omega)).
\end{align*}
As before
\begin{align*}
\psi^\K(\alpha_c + \omega) - \psi^\K(\theta + \omega) &= \partial_\theta \psi^\K(\theta + \omega) (\alpha_c - \theta) + o(\theta - \alpha_c),
\end{align*}
where $\partial_\theta \psi^\K(\theta + \omega) = -A_1(\K,\theta) \cdot (\theta - \alpha_c) + O(1-\K)$ and $\frac1K < A_1(\K, \theta) < K$ for some $\K > 0$, and
\begin{align*}
p(\psi^\K(\alpha_c + \omega)) - p(\psi^\K(\theta + \omega)) &= \left( 1 - 2\psi(\alpha_c + \omega) \right) \partial_\theta \psi^\K(\theta + \omega) (\alpha_c - \theta) +\\
&+ \left( \partial_\theta \psi^\K(\theta + \omega) \right)^2 (\alpha_c - \theta)^2 + o((\alpha_c - \theta)^2) =\\
&= \left( O(1-\K) - (\frac38 + \K\frac58) \right) \partial_\theta \psi^\K(\alpha_c + \omega) (\alpha_c - \theta) + o((\alpha_c - \theta)^2) =\\
&= (\frac38 + \K\frac58) A_1(\K,\theta) \cdot (\theta - \alpha_c)^2 + o(1-\K) + o((\alpha_c - \theta)^2).
\end{align*}
The first term is therefore equal to
\begin{align*}
- A_3(\K, \theta)(\alpha_c - \theta)^2 + o(1-\K),
\end{align*}
for some $\frac1K < A_3(\K, \theta) < K$ (for some $K > 0$), as we have shown above.
The next term satisfies that $c_\K(\alpha_c + \omega) - c_\K(\theta + \omega) = O(\alpha_c - \theta)$ and $p(\psi^\K(\alpha_c + \omega)) = O(1-\K)$.
Therefore
\begin{align*}
\psi^\K(\alpha_c + 2\omega) - \psi^\K(\theta + 2\omega) &= -A_3(\theta, \K)(\alpha_c - \theta)^2 + o((\alpha_c - \theta)(1-\K)) + o(1-\K),
\end{align*}
or, since $\alpha_c - \theta = o(1)$ as $\K \to 1^-$ (they belong to increasingly smaller intervals $I_{n(\K)}$), 
\begin{align*}
\psi^\K(\theta + 2\omega) - \psi^\K(\alpha_c + 2\omega)&= -A_3(\K, \theta)(\alpha_c - \theta)^2 + o(1-\K),
\end{align*}
where $\frac1K < A_3(\K, \theta) < K$, as above.
\end{proof}

\begin{lemma}
For $\theta \in J_\K + 2\omega$, we have that, for $\K < 1$ sufficiently close to 1
\begin{align}
\max_{\theta \in \{\theta + (3 + k)\omega : \theta \in J, 0 \leq k \leq T_1(\K, \theta)\} } |\partial_\theta \psi^\K(\theta)| = \max_{\theta \in \{\theta + (3 + T_1(\K, \theta)) \cdot \omega\} } |\partial_\theta \psi^\K(\theta)|.
\end{align}
and asymptotically, there is a constant $K > 0$, such that 
\begin{align}
\frac1K \cdot \frac{1}{\sqrt{1-\K}} \leq \max_{\theta \in \{J + (3 + T_1(\K, \theta) \cdot \omega\} } |\partial_\theta \psi^\K(\theta)| \leq K \cdot \frac{1}{\sqrt{1-\K}},\label{DerivativeEstimateAtRecoveryPoint}
\end{align}
as $\K \to 1^-$.
\end{lemma}

\begin{proof}
Let $\theta_0 \in J_\K + 2\omega \supseteq I_n + 3\omega$, and set $x_0 = \psi^\K(\theta_0)$. By \cref{DifferenceInValuesAfterPeak}
\begin{align*}
\partial_\theta x_0 = A_2(\K, \theta) (\theta - \alpha_c) + O(1-\K),
\end{align*}
and by \cref{DerivativeAfterPeak}
\begin{align*}
x_0 = \psi^\K(\alpha_c + 2\omega) + A_3(\K, \theta) (\alpha_c - \theta)^2 + o(1-\K).
\end{align*}
Since $\psi^\K(\alpha_c + 2\omega) = K(\K) (1-\K)$, where $\frac1K < K(\K) < K$ (for some $K > 0$), this gives us
\begin{align*}
x_0 = K(\K)(1-\K) + A_3(\K, \theta)(\alpha_c - \theta)^2 + o(1-\K).
\end{align*}
Let $\theta - \alpha_c = L \cdot \sqrt{1-\K}$. By \cref{MayChooseSquareRoot}, it is possible to choose $L$ close to 1. Thus, we have
\begin{align*}
\partial_\theta x_0 = L \cdot A_2(\K, \theta) \sqrt{1 - \K} + o(\sqrt{1-\K}),
\end{align*}
since $O(1-\K) = o(\sqrt{1-\K})$, and
\begin{align*}
x_0 = K(\K)(1-\K) + A_3(\K, \theta) \cdot L^2 \cdot (1-\K) + o(\sqrt{1-\K}).
\end{align*}
Now, by \cref{EachIterationAtBottomAlmostLikeDerivative}, there are constants $0 < D_1 \leq D_2$ such that
\begin{align*}
D_1 \cdot \frac1{x_0} \leq \prod_{k=0}^{T_1(\K, \theta_0)} c(\theta_k) \cdot p'(x_k) = D_2 \cdot \frac1{x_0}.
\end{align*}
Hence, for some $\epsilon > 0$, suppressing the dependence on parameters in the notation of $K, A_2, A_3$,
\begin{align*}
D_1\frac{L \cdot A_2 + \epsilon(\K)}{K + A_3 \cdot L^2 + \epsilon(\K)} \cdot \frac{\sqrt{1-\K}}{1-\K} \leq
|\partial_\theta x_0 \cdot \prod_{k=0}^{T_1(\K, \theta_0)} c(\theta_k) \cdot p'(x_k)| \leq
D_2\frac{L \cdot A_2+ \epsilon(\K)}{K + A_3 \cdot L^2 + \epsilon(\K)} \cdot \frac{\sqrt{1-\K}}{1-\K},
\end{align*}
where $\epsilon(\K) \to 0$ as $\K \to 1^-$.
If $L$ is very big, then $L^2$ would dominate the denominator, and we would have
\begin{align*}
\frac{L \cdot A_2 + \epsilon}{K + A_3 \cdot L^2 + \epsilon} \sim \frac1L.
\end{align*}
If $L$ is very small, then $K$ would dominate the denominator, and we would have
\begin{align*}
\frac{L \cdot A_2 + \epsilon}{K + A_3 \cdot L^2 + \epsilon} \sim L.
\end{align*}
Hence, the the maximum would be obtained if we choose $L$ like $L \sim 1$.

By \cref{SumProductsSmallAfterPeak},
\begin{align*}
\sum \limits_{k=0}^{N-1} \partial_\theta c(\theta_k) \cdot p(x_k) \cdot \prod \limits_{j=k+1}^{N} c(\theta_j) \cdot p'(x_j) = o(x_0^\gamma) = o((1-\K)^\gamma),
\end{align*}
for every $\gamma < 0$. Hence, the derivative will be like
\begin{align*}
const + const_1 \frac{1}{\sqrt{1-\K}} + o(1-\K) \leq |\partial_\theta x_{T_1(\K, \theta_0)}| \leq const + const_2 \frac{1}{\sqrt{1-\K}} + o(1-\K).
\end{align*}
Once the derivative has grown to a certain point, it will grow monotonically (see \cref{DerivativeGrowsDuringExpansion}).
Therefore, as $\K$ gets closer to 1, the derivative must grow past this point, and the maximum would be attained for $|\partial_\theta x_{T_1(\K, \theta_0)}|$.
\end{proof}

This is a good time to remind ourselves that the integers $N_n$ satisfy $\theta_0 \in I_n \Rightarrow \theta_i \not\in I_n$ for $0 \leq i < N_n$. 

\begin{prop}\label{DerivativeGrowth}
Suppose that $0 \leq \K < 1$. Asymptotically, there is a constant $K > 0$, such that 
\begin{align*}
\frac1K \cdot \frac{1}{\sqrt{1-\K}} \leq \max \limits_{\theta \in \mathbb{T}} |\partial_\theta \psi^\K(\theta)| \leq K \cdot \frac{1}{\sqrt{1-\K}},
\end{align*}
as $\K \to 1^-$.
\end{prop}

\begin{proof}
Let $0 \leq \K < 1$ be given, and set $n = n(\K)$, $J = J_\K$.

Recall the definition of $T_1(\K)$ given in \cref{DefinitionOfT1}.
Suppose that $2K_{n-1} - 2 < T_1(\K) \leq (K_{n-1})^{3/2}(2K_{n-1} - 2)$.
Then $T_1(\K) \ll K_{n-1}^3 \sim (N_{m+1})^{3/4}$, and \cref{BigDerivative} implies that $I_{n-1} \subseteq J$.
In this case, set $m = n - 2$, to get
\begin{align*}
K_m^{5/2} \ll K_{m+1} < T_1(\K) \ll K_{m+1}^3 \sim (N_{m+1})^{3/4}.
\end{align*}

Otherwise, if $(K_{n-1})^{3/2}(2K_{n-1} - 2) < T_1(\K) \leq 2K_n - 2$, set $m = n - 1$. By our choice of $m$
\begin{align}
K_m^{5/2} < T_1(\K) \ll K_{m+1}^3 \sim (N_{m+1})^{3/4},\label{T1Bounds}.
\end{align}

Let $\{J + k\omega\}_{k = 0}^M$ be a minimal (in the sense that $M > 0$ is the smallest possible) cover of $\mathbb{T}$.

We know that
\begin{align*}
\max_{\theta \in \{\theta + (3 + k)\omega : \theta \in J, 0 \leq k \leq T_1(\K, \theta)\} } |\partial_\theta \psi^\K(\theta)| =
\max_{\theta \in \{\theta + (3 + T_1(\K, \theta)) \omega\ : \theta \in J\} } |\partial_\theta \psi^\K(\theta)|.
\end{align*}
Therefore, the parts of the cover where we have no control this far is
\begin{align*}
\{\theta + (3 + T_1(\K, \theta) + k) \omega : \theta \in J, 1 \leq k \leq M - 3 + T_1(\K, \theta) \}.
\end{align*}

Pick a $\theta_0 = \theta + (3 + T_1(\K, \theta)) \omega$, where $\theta \in J$. Set $T_1 = T_1(\K, \theta)$ and $x_0 = \psi^\K(\theta_0)$. Suppose that $t \geq 0$ is the smallest integer satisfying
\begin{align*}
x_{t} \in C.
\end{align*}
We wish to get an upper bound on $t$. There are two possibilities; either $\theta_0 \in I_0 \cup (I_0 + \omega)$, or it's not.
In the case $\theta_0 \in I_0 \cup (I_0 + \omega)$, suppose that $\theta_0 \in I_k \backslash I_{k+1} \cup (I_k \backslash I_{k+1} + \omega)$,
where necessarily $k \leq m$ since $T_1 \ll (N_{m+1})^{3/4} < N_{m+1}$.
Then \cref{QuickReturnFromWorstToGood} implies that $x_{2K_k + 20} \in C$, and therefore $t \leq 2K_m + 20$.

In the case $\theta_0 \not\in I_0 \cup (I_0 + \omega)$, there are two possibilities; either $x_t \in C$ for $t \leq 20$, or $\theta_i \in I_0$ for some $i < 20$.
This follows since $\theta_0, \dots, \theta_{19} \not\in I_0 \cup (I_0 + \omega)$ implies that $x_{20} \in C$, by \cref{20IterationsToC}.
Suppose then that $t > 20$, i.e. that $\theta_i \in I_0$, for some $i < 20$, say $\theta_i \in I_k \backslash I_{k+1}$ where $k \leq m$.
It follows that $x_{i + 2K_k + 20} \in C$, or $t \leq i + 2K_k + 20 \leq 2K_m + 39$.

Thus, we obtain the upper bound $t < 3K_m$ on the smallest $t > 0$ satisfying $x_t \in C$. We are now in a position to invoke \cref{LocalControlOnProducts} for $x_t \in C$.
As long as $k \leq N(\theta_0; J)$, this gives us the estimates
\begin{align*}
\prod \limits_{i = 0}^{k-1} |c(\theta_i) \cdot p'(x_i)| = \prod \limits_{i = j}^{t-1} |c(\theta_i) \cdot p'(x_i)| \prod \limits_{i = t}^{k-1} |c(\theta_i) \cdot p'(x_i)| \leq 4^{3K_m} \cdot 4^{4K_m} \cdot (3/5)^{\left(1 - \frac1{M_0} \right)(k-t)/2},
\end{align*}
when $0 \leq j < t$, and
\begin{align*}
\prod \limits_{i = j}^{k-1} |c(\theta_i) \cdot p'(x_i)| \leq \cdot 4^{4K_m} \cdot (3/5)^{\left(1 - \frac1{M_0} \right)(k-j)/2},
\end{align*}
when $t \leq j < k$. Now,
\begin{align*}
\sum \limits_{j = 1}^{k-1} \prod \limits_{i = j}^{k-1} |c(\theta_i) \cdot p'(x_i)| &\leq 4^{7K_m} \cdot \sum \limits_{j = 1}^{k-1} (3/5)^{\left(1 - \frac1{M_0} \right)(k-j)/2} \leq\\
&\leq 4^{7K_m} \cdot \sum \limits_{j = 1}^{\infty-1} (3/5)^{\left(1 - \frac1{M_0} \right)j/2} = 4^{7K_m} \cdot A\\
\end{align*}
where $A > 0$ is some constant, as long as $k \leq N(\theta_0; J)$.

Since $K_m < T_1(\K)^{2/5}$ (see \cref{T1Bounds}), we get $4^{7K_m} = O(4^{2T_1(\K)/5}) = O(\frac1{(1-\K)^{2/5}}) = o(\frac1{\sqrt{1-\K}})$. Therefore
\begin{align*}
|\partial_\theta x_k| &\leq \|\partial_\theta c\| + |\partial_\theta x_0| \cdot \prod \limits_{i=0}^{k - 1} |c(\theta_i) \cdot p'(x_i)| +\\
&+ \|\partial_\theta c\| \sum \limits_{j = 1}^{k-1} \prod \limits_{i = j}^{k-1} |c(\theta_i) \cdot p'(x_i)| \leq\\
&\leq \|\partial_\theta c\| \left( 1 + 4^{7K_m} \cdot A \right) + |\partial_\theta x_0| \cdot 4^{7K_m} \cdot (3/5)^{\left(1 - \frac1{M_0} \right)(k-3K_m)/2} \leq\\
&\leq |\partial_\theta x_0| \cdot o\left( \frac1{\sqrt{1-\K}} \right) + const,
\end{align*}
where the constant satisfies $const = o(\frac1{\sqrt{1-\K}})$ as $\K \to 1^-$, and therefore is negligible.
Since we already have the bounds on $|\partial_\theta x_0|$ in \cref{DerivativeEstimateAtRecoveryPoint}, this gives us the asymptotic inequality
\begin{align*}
\frac1K \cdot \frac{1}{\sqrt{1-\K}} \leq \max \limits_{\theta \in \mathbb{T}} |\partial_\theta \psi^\K(\theta_k)| \leq K \cdot \frac{1}{\sqrt{1-\K}},
\end{align*}
where $K > 0$ as $\K \to 1^-$, as long as $k \leq N(\theta_0, J)$. When $k = N(\theta_0; J)$, we are back in an interval, $J$, where we already know the derivative,
and the derivative of its iterates. We may therefore terminate the process at this point.
\end{proof}

\section*{Acknowledgement}
I want to thank Kristian Bjerklöv for our many valuable discussions during the conception of this article.
This research was partially supported by a Swedish Research Council grant.

\appendix
\section{Some technical lemmas}
In the appendix, we will fix $\K$, and write $c = c_\K$. All the constants are independent of $\K \in [0,1]$, or can be chosen to be independent for these $\K$.
\begin{lemma}\label{MultiplierAndDerivativeAreSimilar}
Suppose that $0 < x_0 < \frac1{100}$, and that $N \geq 0$ is the smallest integer satisfying $\frac1{100} \leq x_{N+1}$. Then
\begin{align*}
\prod \limits_{k=0}^N c(\theta_k) \cdot p'(x_k) = C_N \prod \limits_{k=0}^N c(\theta_k)(1-x_k),
\end{align*}
where $C_N \downarrow C^* > 0$ as $N \to \infty$ (i.e. $x_0 \to 0$).
\end{lemma}

\begin{proof}
We will use \cite[Lemma 15.3]{RudinComplex}, and use the same notation as there.
Since $p'(x) = 1 - 2x$, we see that
\begin{align*}
1 \geq C_N = \frac{\prod \limits_{k=0}^N c(\theta_k) \cdot p'(x_k)}{\prod \limits_{k=0}^N c(\theta_k)(1-x_k)} = \prod \limits_{k=0}^N \frac{1 - 2x_k}{1 - x_k} = \prod \limits_{k=0}^N (1 - \frac{x_k}{1 - x_k}) \geq \prod \limits_{k=0}^N (1 - \gamma x_k) = C_N^* > 0,
\end{align*}
where $1 \leq \gamma \leq \frac1{1 - \max_{0 \leq k \leq N}x_k} = \frac1{1 - x_N}$. We now have that
\begin{align*}
|1 - C_N| \leq C_N^* - 1 \leq \exp( \sum \limits_{k=0}^N \gamma x_k ) - 1.
\end{align*}
Since, for $0 \leq k \leq N$,
\begin{align*}
\frac54 x_k \leq x_{k+1} \leq 4 x_k,
\end{align*}
we see that
\begin{align*}
\frac14 x_{k+1} \leq x_k \leq \frac45 x_{k+1},
\end{align*}
or
\begin{align*}
(\frac14)^{N-k} x_N \leq x_k \leq (\frac45)^{N-k} x_N,
\end{align*}
hence, since $x_N < 1/100$, and therefore $\gamma \leq \frac1{1-x_N} < 100/99$,
\begin{align*}
0 \leq \sum \limits_{k=0}^N \gamma x_k \leq \frac{x_N}{1-x_N} \cdot (\frac45)^N \sum \limits_{k=0}^N (\frac54)^k \leq \frac{x_N}{1-x_N} \cdot 5 < \frac5{99} < \frac1{10}
\end{align*}
So, for every $N \geq 0$,
\begin{align*}
|1 - C_N| \leq \exp(1/10) - 1 \leq 1/5,
\end{align*}
and we conclude that, since $C^* < 1$,
\begin{align*}
C_N \downarrow C^* \geq 4/5.
\end{align*}
\end{proof}

\begin{lemma}\label{EachIterationAtBottomAlmostLikeDerivative}
Suppose that $0 < x_0 < \frac1{100}$, and that $N \geq 0$ is the smallest integer satisfying $\frac1{100} \leq x_{N+1}$.
Then there is a constant $K > 0$ such that
\begin{align*}
\frac1K \cdot \frac1{x_0} \leq \prod \limits_{k=0}^N c(\theta_k) \cdot p'(x_k) \leq K \cdot \frac1{x_0}
\end{align*}
\end{lemma}

\begin{proof}
By \cref{MultiplierAndDerivativeAreSimilar}
\begin{align*}
\prod \limits_{k=0}^{N-1} c(\theta_k) \cdot p'(x_k) = C_N \prod \limits_{k=0}^{N-1} c(\theta_k)(1-x_k),
\end{align*}
where $C_N$ is bounded from below, irrespective of the value of $x_0$. Since
\begin{align*}
x_0 \prod \limits_{k=0}^{N-1} c(\theta_k)(1-x_k) = x_N,
\end{align*}
it follows that
\begin{align*}
\prod \limits_{k=0}^{N-1} c(\theta_k) \cdot p'(x_k) = C_N \cdot x_N \cdot \frac1{x_0}.
\end{align*}
From the assumptions on the bounds of $x_N$, and since $C_N$ is monotonically decreasing (and hence bounded) the statement follows.
\end{proof}

\begin{lemma}\label{SumProductsSmallAfterPeak}
Suppose that $0 < x_0 < \frac1{100}$, and that $N \geq 0$ is the smallest integer satisfying $\frac1{100} \leq x_{N+1}$. Then
\begin{align*}
\sum \limits_{k=0}^{N-1} \partial_\theta c(\theta_k) \cdot p(x_k) \cdot \prod \limits_{j=k+1}^{N} c(\theta_j) \cdot p'(x_j) = o(x_0^\gamma),
\end{align*}
for every $\gamma < 0$.
\end{lemma}

\begin{proof}
\begin{align*}
\sum \limits_{k=0}^{N-1} \partial_\theta c(\theta_k) \cdot p(x_k) \cdot \prod \limits_{j=k+1}^{N} c(\theta_j) \cdot p'(x_j) &=
\sum \limits_{k=0}^{N-1} \partial_\theta c(\theta_k) \cdot (1-x_k)x_k \cdot \prod \limits_{j=k+1}^{N} c(\theta_j) \cdot p'(x_j).
\end{align*}
Since $x_k < x_{k+1}$, we obtain
\begin{align*}
\left| \sum \limits_{k=0}^{N-1} \partial_\theta c(\theta_k) \cdot p(x_k) \cdot \prod \limits_{j=k+1}^{N} c(\theta_j) \cdot p'(x_j) \right| &\leq
\left| \sum \limits_{k=0}^{N-1} \partial_\theta c(\theta_k) \cdot (1-x_k)x_{k+1} \cdot \prod \limits_{j=k + 1}^{N} c(\theta_j) \cdot p'(x_j) \right| = \\
&= \left| \sum \limits_{k=0}^{N-1} \partial_\theta c(\theta_k) \cdot (1-x_k)x_{k+1} \cdot C_{N-k-1} \prod \limits_{j=k + 1}^{N} c(\theta_j) \cdot (1 - x_j) \right|
\end{align*}
Since we had the relation that $N$ is the smallest integer satisfying that
\begin{align*}
\frac1{100} \leq x_{N+1} = \prod \limits_{j=k}^{N} c(\theta_j) \cdot (1 - x_j) x_k \leq \frac4{100} = \frac1{25},
\end{align*}
we obtain the new inequality
\begin{align*}
\left| \sum \limits_{k=0}^{N-1} \partial_\theta c(\theta_k) \cdot p(x_k) \cdot \prod \limits_{j=k+1}^{N} c(\theta_j) \cdot p'(x_j) \right| &\leq
\frac1{25} \cdot C_{N-k-1} \left| \sum \limits_{k=0}^{N-1} \partial_\theta c(\theta_k) \cdot (1-x_k) \right|.
\end{align*}
Since
\begin{align*}
\left| \sum \limits_{k=0}^{N-1} \partial_\theta c(\theta_k) (1 - x_k) \right| &\leq const \cdot N,
\end{align*}
and $N$ is of the order $\log(1/x_0)$, which is of order $o(x_0^{-\gamma})$ for every $\gamma > 0$, the conclusion follows.
\end{proof}

\bibliographystyle{alpha}
\bibliography{references}

\end{document}